\documentclass{article}

\PassOptionsToPackage{numbers,compress,sort}{natbib}
\usepackage[final]{neurips_2022}

\usepackage[utf8]{inputenc} % allow utf-8 input
\usepackage[T1]{fontenc}    % use 8-bit T1 fonts
\usepackage{hyperref}       % hyperlinks
\usepackage{url}            % simple URL typesetting
\usepackage{booktabs}       % professional-quality tables
\usepackage{amsfonts}       % blackboard math symbols
\usepackage{nicefrac}       % compact symbols for 1/2, etc.
\usepackage{microtype}      % microtypography
\usepackage{xcolor}         % colors

\usepackage{bbding}
\usepackage{algorithm}
\usepackage{algorithmic}
\usepackage{graphicx}
\usepackage{enumitem}
\usepackage{mathtools}
\usepackage{wrapfig}
\usepackage{multicol}

\usepackage{amsmath}\allowdisplaybreaks
\usepackage{amsfonts,bm}
\usepackage{amssymb}

\def\eqref#1{equation~(\ref{#1})}
\def\1{\bf{1}}

\newcommand{\Norm}[1]{\left\| #1 \right\|}
\newcommand{\norm}[1]{\left\| #1 \right\|_2}

\def\eps{{\varepsilon}}

\def\vzero{{\bf{0}}}

\def\va{{\bf{a}}}

\def\vg{{\bf{g}}}

\def\vs{{\bf{s}}}

\def\vu{{\bf{u}}}
\def\vv{{\bf{v}}}
\def\vw{{\bf{w}}}
\def\vx{{\bf{x}}}
\def\vy{{\bf{y}}}
\def\vz{{\bf{z}}}

\def\fB{{\mathcal{B}}}

\def\fD{{\mathcal{D}}}

\def\fF{{\mathcal{F}}}

\def\fH{{\mathcal{H}}}

\def\fK{{\mathcal{K}}}

\def\fO{{\mathcal{O}}}

\def\fS{{\mathcal{S}}}
\def\fT{{\mathcal{T}}}

\def\BI{{\mathbb{I}}}

\def\BP{{\mathbb{P}}}

\def\BR{{\mathbb{R}}}
\def\BS{{\mathbb{S}}}

\def\mA {{\bf A}}
\def\mB {{\bf B}}
\def\mC {{\bf C}}

\def\mH {{\bf H}}
\def\mI {{\bf I}}

\def\mT {{\bf T}}
\def\mU {{\bf U}}

\def\mX {{\bf X}}

\def\mZ {{\bf Z}}

\DeclareMathOperator*{\argmax}{arg\,max}
\DeclareMathOperator*{\argmin}{arg\,min}

\usepackage{amsthm}
\usepackage{etoolbox}
\makeatletter
\def\Ddots{\mathinner{\mkern1mu\raise\p@
\vbox{\kern7\p@\hbox{.}}\mkern2mu
\raise4\p@\hbox{.}\mkern2mu\raise7\p@\hbox{.}\mkern1mu}}
\makeatother

\makeatletter
\newcommand*{\rom}[1]{\expandafter\@slowromancap\romannumeral #1@}
\makeatother

\theoremstyle{plain}
\newtheorem{thm}{Theorem}
\AfterEndEnvironment{thm}{\noindent\ignorespaces}
\newtheorem{dfn}{Definition}
\AfterEndEnvironment{defn}{\noindent\ignorespaces}

\AfterEndEnvironment{exmp}{\noindent\ignorespaces}
\newtheorem{lem}{Lemma}
\AfterEndEnvironment{lem}{\noindent\ignorespaces}
\newtheorem{asm}{Assumption}
\AfterEndEnvironment{asm}{\noindent\ignorespaces}
\newtheorem{remark}{Remark}
\AfterEndEnvironment{remark}{\noindent\ignorespaces}
\newtheorem{cor}{Corollary}
\AfterEndEnvironment{cor}{\noindent\ignorespaces}

\AfterEndEnvironment{prop}{\noindent\ignorespaces}

\def\hs {{{\hat \vs}}}
\def\hx {{{\hat \vx}}}
\def\hy {{{\hat \vy}}}

\def\hK {{{\hat K}}}

\def\tg {{{\tilde \vg}}}

\def\tm {{{\tilde m}}}

\def\teps {{{\tilde \eps}}}

\def\AGD {{{\rm AGD}}}

\title{Finding Second-Order Stationary Points in Nonconvex-Strongly-Concave Minimax Optimization}

\author{%
	Luo Luo \\
	School of Data Science \\ 
	Fudan University \\ \texttt{luoluo@fudan.edu.cn}
	\And
	Yujun Li \\
	Noah's Ark Lab \\
	Huawei Technologies Co., Ltd. \\
	\texttt{liyujun9@huawei.com}
	\And
	Cheng Chen\thanks{The corresponding author}\\
 	School of Physical and Mathematical Sciences \\
 	Nanyang Technological University \\
 	\texttt{cheng.chen@ntu.edu.sg}
}

\begin{document}

\maketitle

\begin{abstract}
We study the smooth minimax optimization problem $\min_{\bf x}\max_{\bf y} f({\bf x},{\bf y})$, where $f$ is $\ell$-smooth, strongly-concave in ${\bf y}$ but possibly nonconvex in ${\bf x}$.
Most of existing works focus on finding the first-order stationary points of the function $f({\bf x},{\bf y})$ or its primal function $P({\bf x})\triangleq \max_{\bf y} f({\bf x},{\bf y})$, but few of them focus on achieving second-order stationary points. In this paper, we propose a novel approach for minimax optimization, called Minimax Cubic Newton (MCN), which could find an $\big(\varepsilon,\kappa^{1.5}\sqrt{\rho\varepsilon}\,\big)$-second-order stationary point of $P({\bf x})$ with calling ${\mathcal O}\big(\kappa^{1.5}\sqrt{\rho}\varepsilon^{-1.5}\big)$ times of second-order oracles and $\tilde{\mathcal O}\big(\kappa^{2}\sqrt{\rho}\varepsilon^{-1.5}\big)$ times of first-order oracles, where $\kappa$ is the condition number and $\rho$ is the Lipschitz continuous constant for the Hessian of $f({\bf x},{\bf y})$. In addition, we propose an inexact variant of MCN for high-dimensional problems to avoid calling expensive second-order oracles. Instead, our method solves the cubic sub-problem inexactly via gradient descent and matrix Chebyshev expansion. This strategy still obtains the desired approximate second-order stationary point with high probability but only requires $\tilde{\mathcal O}\big(\kappa^{1.5}\ell\varepsilon^{-2}\big)$ Hessian-vector oracle calls and  $\tilde{\mathcal O}\big(\kappa^{2}\sqrt{\rho}\varepsilon^{-1.5}\big)$ first-order oracle calls.
To the best of our knowledge, this is the first work that considers the non-asymptotic convergence behavior of finding second-order stationary points for minimax problems without the convex-concave assumptions.
\end{abstract}

\section{Introduction} \label{sec:intro}

We consider minimax optimization of the form
\begin{align}\label{prob:main}
    \min_{\vx\in\BR^{d_x}} \max_{\vy\in\BR^{d_y}} f(\vx,\vy),
\end{align}
where $f(\vx,\vy)$ is $\ell$-smooth, $\mu$-strongly-concave in $\vy$, but possibly nonconvex in $\vx$. 
Problem (\ref{prob:main}) can also be written as 
\begin{align}\label{prob:main-primal}
    \min_{\vx\in\BR^{d_x}}\left\{ P(\vx)\triangleq \max_{\vy\in\BR^{d_y}} f(\vx,\vy)\right\}.
\end{align}
This framework covers a wide range of applications in machine learning such as regularized GAN~\cite{sanjabi2018convergence}, reinforcement learning~\cite{qiu2020single}, domain adaptation~\cite{ganin2016domain} and adversarial training~\cite{sinha2017certifying}. 

Most recent works focus on finding an $\eps$-first-order stationary point (FSP) of $P(\vx)$. 
\citet{lin2019gradient} showed that the vanilla gradient descent ascent (GDA) method could obtain an $\eps$-FSP with $\fO((\kappa^2\ell+\kappa\ell^2)\eps^{-2})$ first-order oracle calls. This complexity can be reduced to $\tilde\fO\left(\sqrt{\kappa}\ell\eps^{-2}\right)$ by proximal iteration algorithms~\cite{lin2020near}, which matches the gradient oracle lower bound for finding $\eps$-FSP of $P(\vx)$~\cite{han2021lower,zhang2021complexity}. 
The theory of first-order optimization for problem (\ref{prob:main}) has also been studied in stochastic settings~\cite{lin2019gradient,luo2020stochastic,xu2020gradient,guo2020fast,huang2020accelerated,xian2021faster}. However, the approximate FSP obtained by these algorithms cannot guarantee the local optimality since the primal function $P(\vx)$ could be nonconvex. 

In this paper, we focus on finding a second-order stationary point (SSP) of $P(\vx)$ to capture the local optimal properties~\cite{mazumdar2019finding,fiez2021global}. 
Inspired by the success of second-order optimization in nonconvex minimization~\cite{nesterov2006cubic,zhou2019stochastic,tripuraneni2017stochastic,agarwal2017finding,cartis2011aadaptive,cartis2011badaptive,kohler2017sub,hanzely2020stochastic}, we propose a novel method, called Minimax Cubic
Newton (MCN), which runs cubic Newton update on $\vx$ and maximizes the objective on $\vy$ alternatively. 
This iteration scheme avoids getting stuck at an unexpected FSP. 
Specifically, we show MCN will converge to an $\big(\eps,\kappa^{1.5}\sqrt{\rho\eps}\,\big)$-SSP of $P(\vx)$ with $\fO\big(\kappa^{1.5}\sqrt{\rho}\eps^{-1.5}\big)$ number of iterations, where $\kappa$ is the condition number and $\rho$ is the Lipschitz continuous constant of $\nabla^2f(\vx,\vy)$.
For high-dimensional problems, 
we also propose an efficient algorithm, called  Inexact Minimax Cubic Newton (IMCN), which avoids the expensive second-order oracle calls.
IMCN approximates the second-order information by matrix Chebyshev polynomial and solves the cubic regularized sub-problem inexactly. It only requires $\tilde\fO\big(\kappa^{1.5}\ell\eps^{-2}\big)$ Hessian-vector oracle calls and  $\tilde\fO\big(\kappa^{2}\sqrt{\rho}\eps^{-1.5}\big)$ first-order oracle calls to  find an $\big(\eps,\kappa^{1.5}\sqrt{\rho\eps}\,\big)$-SSP.
Under mild strict saddle condition~\cite{ge2015escaping,ge2016matrix,sun2018geometric,sun2016complete,bhojanapalli2016global}, the approximate SSP of $P(\vx)$ implies an approximate local minimax point of $f(\vx,\vy)$ defined by~\citet{jin2019local}, which successfully characterizes the local optimality for problem (\ref{prob:main}).
To the best of our knowledge, this is the first work that considers non-asymptotic convergence behavior of finding SSP for minimax problems without convex-concave assumptions.
We also conduct experiments on both synthetic function and the real application to validate our theoretical analysis. The empirical results show that the proposed algorithms significantly outperform the GDA method.

%In a concurrent work, \citet{chen2021escaping} also studied nonconvex-strongly-convex minimax problem and proposed a similar algorithm to our MCN and considered the local convergence under {\L}ojasiewicz gradient geometry. However, they did not provide the detailed algorithm for the inexact case like IMCN (Algorithm \ref{alg:IMCN}) in our paper. Additionally, their work do not provide any theoretical analysis for the inexact variant.

In a concurrent work, \citet{chen2021escaping} also studied Problem (\ref{prob:main}) and proposed Cubic-GDA which is similar to our MCN algorithm. MCN has advantage on complexity of first-order oracles by a factor of $\sqrt{\kappa}$ since Cubic-GDA adopts GD to update $\vy$ while MCN uses AGD instead. \citet{chen2021escaping} mentioned that the cubic sub-problem can be efficiently solved by gradient-based algorithms, but they did not provide theoretical analysis for this inexact variant, which is more practical in high dimensional case. As a comparison, we provide the complexity of both Hessian-vector oracles and first-order oracles of our inexact algorithm IMCN.

% \paragraph{Paper Organization}
% We provide preliminaries and relevant backgrounds 
% in Section \ref{sec:pre}. We present our MCN algorithm and its convergence analysis in Section \ref{section:MCN}. Then we provide an inexact variant of MCN as well as its theoretical guarantees in Section \ref{sec:imcn}. We give empirical studies for proposed algorithms in Section \ref{sec:exp}, followed by a conclusion in Section \ref{sec:con}.

\section{Preliminaries} \label{sec:pre}

This section first presents the notations and assumptions for our settings. Then we introduce the background of local optimality for minimax optimization and some basic algorithms. 

\subsection{Notations and Assumptions}

For a twice differentiable function $f(\vx, \vy)$, its partial gradients with respect to $\vx$ and $\vy$ are denoted as $\nabla_x f(\vx, \vy)$ and $\nabla_y f(\vx, \vy)$ respectively. Its Hessian matrix at point $(\vx,\vy)$ can be partitioned as 
$\nabla^2 f(\vx, \vy) = 
[\nabla_{xx}^2 f(\vx, \vy), f(\vx, \vy); \nabla_{yx}^2f(\vx, \vy), \nabla_{yy}^2 f(\vx, \vy)]$,
% \begin{align*}
% \nabla^2 f(\vx, \vy) = \begin{bmatrix}
% \nabla_{xx}^2 f(\vx, \vy) & \nabla_{xy}^2 f(\vx, \vy) \\[0.1cm]
% \nabla_{yx}^2 f(\vx, \vy) & \nabla_{yy}^2 f(\vx, \vy) 
% \end{bmatrix}, 
% \end{align*}
where $\nabla_{xx}^2 f(\vx, \vy)\in\BR^{d_x\times d_x}$, $\nabla_{xy}^2 f(\vx, \vy)\in\BR^{d_x\times d_y}$, $\nabla_{yx}^2 f(\vx, \vy)\in\BR^{d_y\times d_x}$ and $\nabla_{yy}^2 f(\vx, \vy)\in\BR^{d_y\times d_y}$. We also denote 
$\mH(\vx,\vy)=\nabla_{xx}^2 f(\vx, \vy)-\nabla_{xy}^2 f(\vx, \vy)(\nabla_{yy}^2 f(\vx, \vy))^{-1}\nabla_{yx}^2 f(\vx, \vy)$ if $\nabla_{yy}^2 f(\vx, \vy)$ is invertible.

Given a symmetric matrix $\mA$, we denote $\lambda_{\min}(\mA)$ as the smallest eigenvalue of $\mA$.
We use $\norm{\cdot}$ to denote the spectral norm of matrices and Euclidean norm of vectors. We also denote the closed Euclidean ball with radius $r$ and center $\vx^*$ as $\fB(\vx^*,r)=\{\vx:\norm{\vx-\vx^*}\leq r\}$. Additionally, we use notation $\tilde\fO(\cdot)$ to hide logarithmic terms in the complexity.

We suppose the objective function $f(
\vx,\vy)$ of Problem~(\ref{prob:main}) satisfies the following assumptions.

\begin{asm}\label{asm:g-smooth}
The function $f(\vx,\vy)$ has $\ell$-Lipschitz gradients, i.e., there exists a constant $\ell>0$ such that $\norm{\nabla f(\vx,\vy)\!-\!\nabla f(\vx',\vy')}^2 \leq \ell^2\big(\!\norm{\vx-\vx'}^2+\norm{\vy-\vy'}^2\!\big)$
% \begin{align*}
% \norm{\nabla f(\vx,\vy)\!-\!\nabla f(\vx',\vy')}^2 \leq \ell^2\big(\!\norm{\vx-\vx'}^2+\norm{\vy-\vy'}^2\!\big)
% \end{align*}
for any $\vx,\vx'\in\BR^{d_x}$ and $\vy,\vy'\in\BR^{d_y}$.
\end{asm}

\begin{asm}\label{asm:h-smooth}
The function $f(\vx,\vy)$ has $\rho$-Lipschitz Hessian, i.e., there exists a constant $\rho>0$ such that $\norm{\nabla^2 f(\vx,\vy)-\nabla^2 f(\vx',\vy')}^2\leq \rho^2\big(\norm{\vx-\vx'}^2+\norm{\vy-\vy'}^2\big)$
% \begin{align*}
% \norm{\nabla^2 f(\vx,\vy)\!-\!\nabla^2 f(\vx',\vy')}^2\leq \rho^2\big(\!\norm{\vx\!-\!\vx'}^2\!+\!\norm{\vy\!-\!\vy'}^2\!\big)
% \end{align*}
for any $\vx,\vx'\in\BR^{d_x}$ and $\vy,\vy'\in\BR^{d_y}$.
\end{asm}

\begin{asm}\label{asm:concave}
The function $f(\vx,\vy)$ is $\mu$-strongly-concave in $\vy$, i.e., there exists a constant $\mu>0$ such that $f(\vx,\vy)\leq f(\vx,\vy')+\nabla_y f(\vx,\vy)^\top(\vy-\vy')-\frac{\mu}{2}\norm{\vy-\vy'}^2$ for any $\vx\in\BR^{d_x}$ and $\vy,\vy'\in\BR^{d_y}$.
\end{asm}

\begin{asm}\label{asm:lower}
The function $P(\vx)\triangleq\max_{\vy\in\BR^{d_y}} f(\vx,\vy)$ satisfies $P^*\triangleq\inf_{\vx\in\BR^{d_x}} P(\vx)>-\infty$.
% \begin{align*}
% P^*\triangleq\inf_{\vx\in\BR^{d_x}} P(\vx)>-\infty.
% \end{align*}
\end{asm}
\begin{dfn}
Under Assumption \ref{asm:g-smooth} and \ref{asm:concave}, we define the condition number of $f(\vx,\vy)$ as $\kappa\triangleq\ell/\mu$.
\end{dfn}

The assumptions of Lipschitz continuous gradient and strongly-concavity on $f$ indicate that the primal function $P(\vx)\triangleq\max_{\vy\in\BR^{d_y}} f(\vx, \vy)$ is well-defined and has Lipschitz continuous gradients as shown in Lemma~\ref{lem:P-smooth}.

\begin{lem}[{\cite[Lemma 4.3]{lin2019gradient}}]\label{lem:P-smooth}
Suppose the objective function $f$ satisfies Assumptions \ref{asm:g-smooth} and \ref{asm:concave}, then the primal function $P(\vx)\triangleq\max_{\vy\in\BR^{d_y}} f(\vx, \vy)$ has
$(\kappa+1)\ell$-Lipschitz continuous gradients. Additionally, the function $\vy^*(\vx)=\argmax_{\vy\in\BR^{d_y}} f(\vx, \vy)$ is well-defined and $\kappa$-Lipschitz. We also have $\nabla P(\vx) = \nabla_\vx f(\vx, \vy^*(\vx))$.
% \begin{align*}
% \nabla P(\vx) = \nabla_\vx f(\vx, \vy^*(\vx)).    
% \end{align*}
\end{lem}
Now we give the definitions of $\eps$-FSP and $(\eps,\delta)$-SSP as follows. 

\begin{dfn}\label{dfn:approx-fsp}
Suppose the function $f(\vx,\vy)$ satisfies Assumption~\ref{asm:g-smooth}~and~\ref{asm:concave},
then we call $\vx$ an $\eps$-FSP of $P(\vx)$ if $\norm{\nabla{P(\vx)}}\leq\eps$.
\end{dfn}

%Then we define $(\eps,\delta)$-second-order-stationary point (SSP). 

\begin{dfn}\label{dfn:approx-ssp}
Suppose the function $f(\vx,\vy)$ satisfies Assumption \ref{asm:g-smooth}, \ref{asm:h-smooth} and \ref{asm:concave}, then we call $\vx$ is an $(\eps,\delta)$-SSP of $P(\vx)$ if $\norm{\nabla P(\vx)}\leq\eps$ and $\nabla^2 P(\vx)\succeq -\delta \mI$.
\end{dfn}

The following two lemmas provide the closed form of $\nabla^2 P(\vx)$ and its Lipschitz continuity.
\begin{lem}[\cite{shapiro1985second}]\label{lem:P-Hessian}
Suppose the function $f(\vx,\vy)$ satisfies Assumption \ref{asm:g-smooth}, \ref{asm:h-smooth} and \ref{asm:concave}.
We use the definition of $\vy^*(\vx)$ in Lemma \ref{lem:P-smooth}, then it holds that $\nabla^2 P(\vx) = \mH(\vx,\vy^*(\vx))$.
\end{lem}
\begin{lem}\label{lem:Ph-smooth}
Under assumptions of Lemma~\ref{lem:P-Hessian}, we have $\norm{\nabla^2 P(\vx) - \nabla^2 P(\vx')} \leq 4\sqrt{2}\kappa^3\rho\norm{\vx-\vx'}$
% \begin{align*}
%  \norm{\nabla^2 P(\vx) - \nabla^2 P(\vx')} \leq 4\sqrt{2}\kappa^3\rho\norm{\vx-\vx'}.
% \end{align*}
for any $\vx$ and $\vx'$ in $\BR^{d_x}$.
\end{lem}

\vspace{-0.2cm}
\subsection{Local Optimality of Minimax Optimization}

The Nash equilibrium is widely used in the study of convex-concave minimax optimization~\cite{wang2019solving,zhang2020newton}, but it is intractable in general when the objective function $f(\vx,\vy)$ is nonconvex in $\vx$ or nonconcave in $\vy$. 
For the general minimax problem, we introduce the local minimax point~\cite{jin2019local}, which characterizes the optimality in two-player sequential games where players can only change their strategies locally. 

\begin{dfn}[\cite{jin2019local}]\label{dfn:local-minimax}
Given a differentiable function $f(\vx,\vy):\BR^{d_x}\times\BR^{d_y}\to\BR$ that is strongly-concave in $\vy$, a point $(\vx^*, \vy^*)\in\BR^{d_x}\times\BR^{d_y}$ is called a local minimax point of $f$, if there exists $\delta_0 > 0$ and a function $h$ satisfying $h(\delta) \to 0$ as $\delta \to 0$, such that for any $\delta\in(0, \delta_0]$, $\vx\in \fB(\vx^*,h(\delta))$ and $\vy\in\BR^{d_y}$, we have 
\begin{align}\label{cond:local-minimax}
f(\vx^*, \vy) \leq f(\vx^*, \vy^*) \leq \max_{\vy'\in\BR^{d_y}} f(\vx, \vy').
\end{align}
\end{dfn}
\begin{remark}
The definition of local minimax point for general nonconvex-nonconcave~\cite{jin2019local} only requires (\ref{cond:local-minimax}) holds for any $\vy$ and $\vy'$ in a neighbour of $\vy^*$, while the constraint on $\vy$ and $\vy'$ is unnecessary in our setting since we assume that $f(\vx,\vy)$ is strongly-concave in $\vy$.
\end{remark}
The local minimax point enjoys the following property.
\begin{lem}[{\cite[Proposition 19]{jin2019local}}]\label{lem:local-sufficient}
Suppose $f(\vx,\vy)$ is twice differentiable, then any point $(\vx^*,\vy^*)$ satisfying $\nabla f(\vx^*, \vy^*)=\vzero$, $\nabla_{yy}^2 f(\vx^*, \vy^*)\prec \vzero$ and $\mH(\vx^*,\vy^*)\succ0$
is a local minimax point of~$f(\vx,\vy)$. 
\end{lem}
Based on Lemma~\ref{lem:local-sufficient},
we introduce the strict-minimax condition on $f$ which is an extension of strict-saddle condition for nonconvex minimization~\cite{ge2015escaping,ge2016matrix,sun2018geometric,sun2016complete,bhojanapalli2016global,jin2017escape}.
\begin{dfn}\label{dfn:strict}
Under Assumption~\ref{asm:g-smooth},~\ref{asm:h-smooth}~and~\ref{asm:concave}, we say $f(\vx,\vy)$ is $(\alpha, \beta, \gamma)$-strict-minimax for some $\alpha>0$, $\beta>0$ and $\gamma>0$ if every $(\hx,\hy)\in \BR^{d_x}\times\BR^{d_y}$ satisfies at least one of the following three conditions:
(a) $\norm{\nabla {f(\hx,\hy)}}>\alpha$; 
(b) $\lambda_{\min}(\mH(\hx,\hy))< -\beta$;  
(c) There exists a local minimax point $(\vx^*, \vy^*)$ such that $\norm{\hx-\vx^*}^2+\norm{\hy-\vy^*}^2\leq\gamma^2$.
\end{dfn}
Note that if $f$ is $(\alpha, \beta, \gamma)$-strict-minimax and there exists a point $\hx$ which is an $(\eps,\delta)$-SSP of $P(\vx)$ with sufficient small $\eps$ and $\delta$, then we can find $\hy \approx \argmax_\vy f(\hx,\vy)$ via running a first-order algorithm to minimize $-f(\hx, \cdot)$ and obtain $(\hx, \hy)$ which is in a neighborhood of $(\vx^*,\vy^*)$. In other words, under the strict-minimax condition, we can reduce the task of finding approximate local minimax point of $f(\vx,\vy)$ to finding an approximate SSP of $P(\vx)$. We will provide the formal statement in Section \ref{subsection:convergence}.

\subsection{Accelerated Gradient Descent}

\begin{wrapfigure}[9]{R}{0.5\textwidth}
\vspace{-0.9cm}
    \begin{minipage}{0.5\textwidth}
     \begin{algorithm}[H]
\caption{$\AGD (h, \vy_0, K, \eta, \theta)$} \label{alg:agd} 
\begin{algorithmic}[1]
\STATE $\tilde \vy_0=\vy_0$ \\[0.05cm]
\STATE\textbf{for} $k=0,\dots,K-1$ \textbf{do} \\[0.05cm]
\STATE\quad $\vy_{k+1}={\tilde \vy}_k-\eta\nabla h(\tilde \vy_k)$ \\[0.1cm]
\STATE\quad ${\tilde \vy}_{k+1} = \vy_{t+1}+\theta(\vy_{k+1}-\vy_k)$ \\[0.1cm]
\STATE\textbf{end for} \\[0.05cm]
\STATE\textbf{Output:} $\vy_K$
\end{algorithmic}
\end{algorithm}
    \end{minipage}
  \end{wrapfigure}

Nesterov’s accelerated gradient descent (AGD) is the optimal first-order algorithm for convex optimization~\cite{nesterov1983method,nesterov2018lectures}, which is widely used in minimax optimization algorithms~\cite{lin2020near,wang2020improved}. We describe the details of AGD for smooth and strongly-convex functions in Algorithm~\ref{alg:agd}, which has the following convergence rate.

\begin{lem}[{\cite[Lemma 2]{wang2020improved}}]\label{lem:agd}
Running Algorithm~\ref{alg:agd} on a $\ell_h$-smooth and $\mu_h$-strongly-convex objective function $h(\cdot)$ with parameters $\eta=1/\ell_h$ and $\theta=\frac{\sqrt{\kappa_h}-1}{\sqrt{\kappa_h}+1}$ produces the output $\vy_K$ satisfying
$\norm{\vy_K-\vy^*}^2 \leq (\kappa_h +1)\big(1-\frac{1}{\sqrt{\kappa_h}}\big)^K\norm{\vy_0-\vy^*}^2$,
where $\vy^*=\argmin_y h(\vy)$ and $\kappa_h=\ell_h/\mu_h$.
\end{lem}

\vspace{-0.5cm}
\subsection{Cubic Regularized Newton}
Cubic regularized Newton (CRN) is a classic algorithm for nonconvex minimization~\cite{nesterov2006cubic,nesterov2018lectures,cartis2011aadaptive,cartis2011badaptive,tripuraneni2017stochastic}. 
It solves the nonconvex minimization problem 
$\min_\vx g(\vx)$ via the following update rules
%{\small\begin{align*}
  \[\vs_{t}=\argmin_{\vs\in\BR^{d}}\nabla g(\vx_t)^\top \vs+ \frac{1}{2}\vs^\top \nabla^2 g(\vx_t) \vs+\frac{\rho_g}{6}\norm{\vs}^3, ~~~~~~ \vx_{t+1}=\vx_t + \vs_{t}.  \]
%\end{align*}}\\
The CRN method could find an $\big(\eps,\sqrt{\rho_g\eps}\,\big)$-SSP of $g(\vx)$ with $\fO\big(\sqrt{\rho_g}\eps^{-1.5}\big)$ number of iterations, where $\rho_g$ is the Lipschitz continuous constant of the Hessian of $g(\vx)$.
% \begin{align*}
% \begin{dcases}
% \vs_{t}=\argmin_{\vs\in\BR^{d}}\nabla g(\vx_t)^\top \vs+ \frac{1}{2}\vs^\top \nabla^2 g(\vx_t) \vs+\frac{\rho_g}{6}\norm{\vs}^3,\\[0.1cm]
% \vx_{t+1}=\vx_t + \vs_{t}.
% \end{dcases}
% \end{align*}

\section{Minimax Cubic Newton Algorithm}\label{section:MCN}

In this section, we propose our minimax cubic Newton algorithm and give its convergence results.

\subsection{Minimax Cubic Newton Method}

\begin{wrapfigure}[13]{R}{0.5\textwidth}
\vspace{-0.35cm}
\begin{minipage}{0.5\textwidth}
\begin{algorithm}[H]
\caption{Minimax Cubic-Newton (MCN)}\label{alg:MCN}
\begin{algorithmic}[1]
    \STATE \textbf{Input:} $\vx_0\in\BR^{d_x}$, $\vy_{-1}=\vzero$, $T$, $\{K_t\}_{t=0}^T$, $\eps$.  \\[0.05cm]
    \STATE \textbf{for} $t=0, \cdots T-1$ \textbf{do} \\[0.05cm]
    \STATE~~\label{line:cubic-AGD} $\vy_{t} = \AGD\left(-f(\vx_t,\cdot), \vy_{t-1}, K_t, \frac{1}{\ell}, \frac{\sqrt{\kappa}-1}{\sqrt{\kappa}+1}\right)$ \\[0.05cm]
    \STATE~~$\vg_t=\nabla_x f(\vx_t, \vy_t)$, $\mH_t=\mH(\vx_t, \vy_t)$ \\[0.05cm]
    \STATE\label{line:cubic-solver}~~$\vs_t^* = \argmin\limits_{\vs\in\BR^{d_x}}\left(\vg_t^\top \vs + \frac{1}{2}\vs^\top \mH_t \vs + \frac{M}{6}\norm{\vs}^3\right)$ \\[0.05cm]
    \STATE\label{line:break}~~\textbf{if} $\norm{\vs_t^*} \leq \frac{1}{2}\sqrt{\eps/M}$ \textbf{then} break \\[0.05cm]
    \STATE~~$\vx_{t+1} = \vx_t + \vs_t^*$ \\[0.05cm]
    \STATE\textbf{end for} \\[0.05cm]
    \STATE \textbf{Output:} $\hx = \vx_t+\vs_t^*$
\end{algorithmic}
\end{algorithm}
\end{minipage}
\end{wrapfigure}

We present the details of Minimax Cubic Newton (MCN) method in Algorithm~\ref{alg:MCN}. In each round, the MCN algorithm performs following steps:

\begin{itemize}[leftmargin=*]
\item Run AGD as presented in Algorithm~\ref{alg:agd} to estimate $\vy_t \approx \vy^*(\vx_t)=\argmax_\vy f(\vx_t,\vy)$.
\item Compute the inexact first-order and second-order information of $P$ at $\vx_t$ as 
\begin{align*}
\nabla P(\vx_t) \approx  \vg_t = \nabla_x f(\vx_t, \vy_t) 
\quad\text{and}\quad
\nabla^2 P(\vx_t) \approx \mH_t = \mH(\vx_t, \vy_t).    
\end{align*}
\item Solve the following cubic regularized problem
\begin{align}\label{prob:sub-cubic} \hspace{-0.3cm}
\vs_t^* {=} \argmin_{\vs\in\BR^{d_x}}\left(\vg_t^\top + \frac{1}{2}\vs^\top \mH_t  + \frac{M}{6}\norm{\vs}^3\right).
\end{align}
\end{itemize}

The expressions of $\vg_t$ and $\mH_t$ in the algorithm are inspired from Lemma~\ref{lem:P-smooth} and \ref{lem:Ph-smooth}. The smoothness of $\nabla P(\vx)$ and $\nabla^2 P(\vx)$ allow the total complexity of AGD iteration of the algorithm has the desired upper bound.
Our theoretical analysis show the termination condition in Line 7 of Algorithm \ref{alg:MCN} can be attained in no more than $\fO(\kappa^{1.5}\sqrt{\rho}\eps^{-1.5})$ number of iterations. We also show that a small $\norm{\vs_t^*}$ will lead to a desired approximate second-order stationary point of $P(\vx)$. Hence, using AGD to optimize $f(\vx_t+\vs_t^*,\vy)$ with respect to $\vy$ generates an approximate local minimax point of $f(\vx,\vy)$.

\subsection{Complexity Analysis for MCN}\label{subsection:convergence}

In this section, we let $M=4\sqrt{2}\kappa^3\rho$ for MCN. % Algorithm~\ref{alg:MCN}.
Our analysis for MCN algorithm contains three parts:
\begin{enumerate}[leftmargin=*]
\item We follow~\citet{tripuraneni2017stochastic}'s idea to show that our algorithm with sufficient accurate gradient and Hessian estimator of $P(\vx)$ requires no more than $T=\fO\left(\kappa^{1.5}\sqrt{\rho}\eps^{-1.5}\right)$ rounds of iterations to achieve an $\left(\eps, \kappa^{1.5}\sqrt{\rho\eps}\,\right)$-SSP of $P(\vx)$.
\item We prove the AGD step in line 3 requires at most $\tilde\fO\left(\kappa^2\sqrt{\rho}\eps^{-1.5}\right)$ gradient calls in total.
\item The last part shows we can achieve an approximate local minimax point of $f(\vx,\vy)$ from an $\left(\eps, \kappa^{1.5}\sqrt{\rho\eps}\,\right)$-second-order stationary point of $P(\vx)$ under strict-saddle condition.
\end{enumerate}

\paragraph{Cubic Newton Iteration on $P(\vx)$}

The procedure of our Algorithm~\ref{alg:MCN} can be regarded as applying cubic Newton method to minimize nonconvex function $P(\vx)$, but using inexact first-order and second-order information.
We consider the following conditions on the inexact gradient and Hessian, which will hold if we run AGD with enough number of iterations.
\begin{asm}\label{asm:error-g-h}
Suppose the estimators $\vg_t\in\BR^{d_x}$ and $\mH_t\in\BR^{d_x\times d_x}$ satisfy conditions $\norm{\nabla P(\vx_t)-\vg_t} \leq C_g\eps$ and
$\norm{\nabla^2 P(\vx_t)-\mH_t} \leq C_H\sqrt{M\eps}$
for some $C_g> 0$ and $C_H> 0$.
\end{asm}

The following lemma implies the analysis of MCN algorithm only needs to focus on the case of each $\norm{\vs_t^*}$ is large, otherwise we have already find $\vx_{t+1}$ which a desired approximate SSP of $P(\vx)$.
\begin{lem}\label{lem:norm-s}
Under Assumption~\ref{asm:error-g-h} with $C_g=1/192$ and $C_H=1/48$, if $\vx_{t+1}$ from Algorithm~\ref{alg:MCN} is not an $\big(\eps, \sqrt{M\eps}\,\big)$-SSP of $P(\vx)$, we have
$\norm{\vs_t^*} \geq \frac{1}{2}\sqrt{\eps/M}$.
\end{lem}

By Lemma \ref{lem:norm-s}, we can show that MCN with sufficient accurate gradient and Hessian estimators can find an $\big(\eps, \sqrt{M\eps}\,\big)$-SSP of $P(\vx)$ with $T=\fO\big(\kappa^{1.5}\sqrt{\rho}\eps^{-1.5}\big)$ number of iterations as follows.
\begin{thm}\label{thm:MCN}
Under Assumption~\ref{asm:g-smooth}-\ref{asm:lower}, we run Algorithm~\ref{alg:MCN} with $T=\big\lceil 192(P(\vx_0) - P^*)\sqrt{M}\eps^{-1.5}\big\rceil+1$; and further suppose $K_t$ is sufficient large so that results $\vg_t$ and $\mH_t$ satisfy Assumption~\ref{asm:error-g-h} with $C_g=1/192$ and $C_H=1/48$. Then the output $\hx$ is an $\big(\eps, \sqrt{M\eps}\,\big)$-SSP of $P(\vx)$.
\end{thm}

\paragraph{Total Complexity of AGD}
Note that MCN (Algorithm~\ref{alg:MCN}) applies AGD to maximize $f(\vx_t,\cdot)$ by using $\vy_{t-1}$ as initialization at the $t$-th round. With such initialization, the following theorem provides the upper bound of total number of gradient calls required by AGD and guarantees that $\vg_t$ and $\mH_t$ in the MCN algorithm satisfy Assumption~\ref{asm:error-g-h}.
\begin{thm}\label{thm:sum_K}
Under Assumption~\ref{asm:g-smooth}-\ref{asm:lower}, we run Algorithm~\ref{alg:MCN} with $K_0=\big\lceil2\sqrt{\kappa}\log\big(\frac{\sqrt{\kappa+1}}{\teps}\norm{\vy^*(\vx_0)}\big)\big\rceil$ and $K_t=\big\lceil2\sqrt{\kappa}\log\big(\frac{\sqrt{\kappa+1}}{\teps}\big(\teps + \kappa\norm{\vs_{t-1}^*}\big)\big)\big\rceil$ for $t\geq 1$,
% \begin{align*}
% K_t = 
% \begin{cases}
% \left\lceil2\sqrt{\kappa}\log\left(\frac{\sqrt{\kappa+1}}{\teps}\norm{\vy^*(\vx_0)}\right)\right\rceil ,& t = 0 \\[0.3cm]
% \left\lceil2\sqrt{\kappa}\log\left(\frac{\sqrt{\kappa+1}}{\teps}\Big(\teps + \kappa\norm{\vs_{t-1}}\Big)\right)\right\rceil,  & t\geq 1
% \end{cases}
% \end{align*}
where $\teps=\min\big\{C_g\eps/\ell, C_H\sqrt{M\eps}/\rho\big\}$ and $\vs_0=\vzero$. Then all $\vg_t$ and $\mH_t$ satisfy the condition of Assumption~\ref{asm:error-g-h}. We also have
{\small\begin{align*} \vspace{-0.3cm}
\sum_{t=0}^{T} K_t
\leq  T {+} 1 {+} \frac{2\sqrt{\kappa}T}{3}\Big[3\log\Big(\frac{\sqrt{\kappa+1}}{\teps}\norm{\vy^*(\vx_0)}\Big){+}\log\Big(8(\kappa{+}1)^{1.5}{+} \frac{8\kappa^3(\kappa+1)^{1.5}}{T\teps^3} \sum_{t=1}^{T}\norm{\vs_{t-1}^*}^3\Big)\Big].%\\[-0.5cm]
\end{align*}}
\end{thm}

\begin{minipage}{0.48\textwidth}
\begin{algorithm}[H]\small
\caption{Inexact Minimax Cubic-Newton}\label{alg:IMCN}
\begin{algorithmic}[1]
    \STATE \textbf{Input:} $\vx_0\in\BR^{d_x}$, $\vy_{-1}=\vzero$, $T$, $\{K_t\}_{t=0}^T$, $\eps$  \\
    \STATE $c_k = \frac{2}{\sqrt{\ell \mu}}\left(\frac{\sqrt{\mu/\ell}-1}{\sqrt{\mu/\ell}+1}\right)^k$\\ 
    \STATE \textbf{for} $t=0, \cdots$ \textbf{do} \\
    \STATE\quad $\vy_{t} = \AGD\big(-f(\vx_t,\cdot), \vy_{t-1}, K_t, \frac{1}{\ell}, \frac{\sqrt{\kappa}-1}{\sqrt{\kappa}+1}\big)$ \\
    \STATE\quad $\vg_t=\nabla_x f(\vx_t, \vy_t)$ \\[0.03cm]
    \STATE\quad $\mZ_t=-\frac{2}{\ell-\mu}\left(\nabla_{yy}^2 f(\vx_t,\vy_t)+\frac{\ell+\mu}{2}\mI\right)$ \\[0.03cm]
    \STATE\quad Compute $\mH_t$ as equation (\ref{eq:Ht-poly})\\[0.03cm]
    \STATE\quad $(\vs_t, \Delta_t) {=} \text{Cubic-Solver}(\vg_t, \mH_t, \sigma, \fK\left(\eps, \delta'\right))$ \\[0.03cm]
    \STATE\quad $\vx_{t+1} = \vx_t + \vs_t$ \\[0.03cm]
    \STATE\quad\label{line:cond-final} \textbf{if} $\Delta_t > -\frac{1}{128}\sqrt{\eps^3/M}$ \textbf{then}  \\[0.05cm]
    \STATE\quad\quad $\hs = \text{Final-Cubic-Solver}(\vg_t, \mH_t, \eps)$ %\\[0.1cm]
    \STATE\quad\quad $\vx_{t+1} = \vx_t + \hs$ %\\[0.1cm]
    \STATE\quad\quad \textbf{break} %\\[0.05cm]
    \STATE\quad\textbf{end if} %\\[0.05cm]
     \STATE\textbf{end for} %\\[0.05cm]
    \STATE \textbf{Output:} $\hx=\vx_{t+1}$ %and $\hy = \AGD\left(-f(\hx,\cdot), \vy_{t}, K', \frac{1}{\ell}, \frac{\sqrt{\kappa}-1}{\sqrt{\kappa}+1}\right)$
\end{algorithmic}
\end{algorithm}\vskip0.2cm
\end{minipage}
\hspace{0.2cm}
\begin{minipage}{0.48\textwidth}
\begin{algorithm}[H]\small
\caption{Cubic-Solver}\label{alg:cubic-GD}
\begin{algorithmic}[1]
    \STATE \textbf{Input:} $\vg$, $\mH$, $\sigma$, $\fK(\eps,\delta')$  \\[0.1cm]
    \STATE \textbf{if} $\norm{\vg}\geq L^2/M$ \textbf{then}  \\[0.1cm]
    \STATE\quad $R_C{=}-\dfrac{\vg^\top \mH\vg}{M\norm{\vg}^2} {+} \sqrt{\left(\dfrac{\vg^\top \mH \vg}{M\norm{\vg}^2}\right)^2 {+} \dfrac{2\norm{\vg}}{M}}$\\[0.1cm]
    \STATE\quad $\hs = -R_C \vg/\norm{\vg}$ \\[0.1cm]
    \STATE \textbf{else}  \\[0.1cm]
    \STATE\quad $\vs_0 = \vzero$, $\eta = 1/(20L)$ \\[0.1cm]
    \STATE\quad $\tg = \vg + \sigma{\bm\zeta}$ for ${\bm\zeta}\sim {\rm Uniform}(\BS^{d-1})$ \\[0.1cm]
    \STATE\quad \textbf{for} $k=0, \cdots, \fK(\eps,\delta')$ \textbf{do} \\[0.1cm]
    \STATE\quad\quad $\vs_{k+1} = \vs_k - \eta\left(\tg + \mH\vs_k + \frac{M}{2}\norm{\vs_k}\vs_k\right)$ \\[0.1cm]
    \STATE\quad \textbf{end for} \\ [0.1cm]
    \STATE\quad $\hs=\vs_{\fK(\eps,\delta')}$ \\ [0.1cm]
    \STATE \textbf{end if}\\
    \STATE \textbf{Output:} $\hs$ and $\Delta = \vg^\top \hs + \frac{1}{2}\hs^\top \mH \hs + \frac{M}{6}\norm{\hs}^3$
\end{algorithmic}
\end{algorithm}\vskip0.2cm
\end{minipage}

Combining Theorem~\ref{thm:MCN} and Theorem~\ref{thm:sum_K}, we can obtain the total number of gradient oracle calls, Hessian (inverse) oracle calls and exact cubic sub-problem solver calls as follows.

% \begin{algorithm}[t]
% \caption{Inexact Minimax Cubic-Newton (IMCN)}\label{alg:IMCN}
% \begin{multicols}{2}
% \begin{algorithmic}[1]
%     \STATE \textbf{Input:} $\vx_0\in\BR^{d_x}$, $\vy_{-1}=\vzero$, $T$, $\{K_t\}_{t=0}^T$, $\eps$  \\
%      \STATE \textbf{for} $t=0, \cdots$ \textbf{do} \\[0.1cm]
%     \STATE\quad $\vy_{t} = \AGD\left(-f(\vx_t,\cdot), \vy_{t-1}, K_t, \frac{1}{\ell}, \frac{\sqrt{\kappa}-1}{\sqrt{\kappa}+1}\right)$ \\[0.1cm]
%     \STATE\quad $\vg_t=\nabla_x f(\vx_t, \vy_t)$ \\[0.1cm]
%     \STATE\quad $c_k = \frac{2}{\sqrt{\ell \mu}}\left(\frac{\sqrt{\mu/\ell}-1}{\sqrt{\mu/\ell}+1}\right)^k$ \\[0.1cm]
%     \STATE\quad $\mZ_t=-\frac{2}{\ell-\mu}\left(\nabla_{yy}^2 f(\vx_t,\vy_t)+\frac{\ell+\mu}{2}\mI\right)$ \\[0.1cm]
%     \STATE\quad Compute $\mH_t$ as equation (\ref{eq:Ht-poly})\\[0.1cm]
%     \STATE\quad $(\vs_t, \Delta_t) = \text{Cubic-Solver}(\vg_t, \mH_t, \sigma, \fK\left(\eps, \delta'\right))$ \\[0.1cm]
%     \STATE\quad $\vx_{t+1} = \vx_t + \vs_t$ \\[0.1cm]
%     \STATE\quad\label{line:cond-final} \textbf{if} $\Delta_t > -\frac{1}{128}\sqrt{\eps^3/M}$ \textbf{then}  \\[0.1cm]
%     \STATE\quad\quad $\hs = \text{Final-Cubic-Solver}(\vg_t, \mH_t, \eps)$ \\[0.1cm]
%     \STATE\quad\quad $\vx_{t+1} = \vx_t + \hs$ \\[0.1cm]
%     \STATE\quad\quad \textbf{break} \\[0.05cm]
%     \STATE\quad\textbf{end if} \\[0.05cm]
%      \STATE\textbf{end for} \\[0.05cm]
%     \STATE \textbf{Output:} $\hx=\vx_{t+1}$ %and $\hy = \AGD\left(-f(\hx,\cdot), \vy_{t}, K', \frac{1}{\ell}, \frac{\sqrt{\kappa}-1}{\sqrt{\kappa}+1}\right)$
% \end{algorithmic}
% \end{multicols}
% \end{algorithm}

\begin{cor}\label{cor:complexity}
Under Assumption~\ref{asm:g-smooth}-\ref{asm:lower}, running Algorithm~\ref{alg:MCN} with $T{=}\big\lceil 192(P(\vx_0) - P^*)\sqrt{M}\eps'^{-1.5}\big\rceil+1$, $K_0{=}\big\lceil2\sqrt{\kappa}\log\big(\frac{\sqrt{\kappa+1}}{\teps'}\norm{\vy^*(\vx_0)}\big)\big\rceil$ and $K_t{=}\big\lceil2\sqrt{\kappa}\log\big(\frac{\sqrt{\kappa+1}}{\teps'}\big(\teps' + \kappa\norm{\vs_{t-1}}\big)\big)\big\rceil$ for $t \geq 1$,
% {\small\begin{align*}
% T=\big\lceil 192(P(\vx_0) - P^*)\sqrt{M}\eps'^{-1.5}\big\rceil+1 
% \quad\text{and}\quad
% K_t=\begin{cases}
% \left\lceil2\sqrt{\kappa}\log\left(\frac{\sqrt{\kappa+1}}{\teps'}\norm{\vy^*(\vx_0)}\right)\right\rceil ,& t = 0 \\[0.3cm]
% \left\lceil2\sqrt{\kappa}\log\left(\frac{\sqrt{\kappa+1}}{\teps'}\Big(\teps' + \kappa\norm{\vs_{t-1}}\Big)\right)\right\rceil,  & t\geq 1
% \end{cases} 
% \end{align*}}
where $\teps'{=}\min\big\{C_g\eps'/\ell, C_H\sqrt{M\eps'}/\rho\big\}$, $\vs_0{=}\vzero$ and $\eps'{=}2^{-2.5}\eps$,
then the output $\hx$ is an $\big(\eps,\kappa^{1.5}\sqrt{\rho\eps}\,\big)$-SSP of $P(\vx)$ and the number of gradient oracle calls is at most $\tilde\fO\left(\kappa^2\sqrt{\rho}\eps^{-1.5}\right)$. The total number of Hessian (inverse) oracle calls and exact cubic sub-problem solver calls is at most $\fO\left(\sqrt{\rho}\kappa^{1.5}\eps^{-1.5}\right)$.
\end{cor}

Note that MCN method needs to construct the Hessian estimator $\mH_t=\mH(\vx_t,\vy_t)$
and solves the cubic regularized sub-problem~(\ref{prob:sub-cubic}) in each round. Constructing $\mH_t$ requires calling the second-order oracle at $(\vx_t, \vy_t)$ and taking $\fO\left(d_y^3+d_xd_y^2+d_x^2d_y\right)$ flops for matrix multiplication and inversion. Solving sub-problem~(\ref{prob:sub-cubic}) requires $\fO\left(d_x^3\right)$ flops for matrix factorization or inversion~\cite{cartis2011aadaptive,cartis2011badaptive}. 
Hence, besides the gradient calls from AGD step, MCN requires $\fO\left(d_x^3+d_y^3\right)$ time complexity in each round and its space complexity is $\fO\left(d_x^2+d_y^2\right)$.

\paragraph{Approximate Local Minimax Point} Under the strict-minimax condition, we can find an approximate local minimax point by performing an additional AGD procedure on the output of  MCN.

\begin{cor}\label{cor:local-minimax}
Suppose $f(\vx,\vy)$ is $(\alpha,\beta,\gamma)$-strict-minimax and satisfies Assumption~\ref{asm:g-smooth}-\ref{asm:lower}. Let $\hx$ be the output of running Algorithm~\ref{alg:MCN} with the setting of Corollary~\ref{cor:complexity} and $\eps=\min\left\{\alpha/3,~\beta^2/(8\kappa^3\rho)\right\}$.
Let $\hy=\AGD\big(-f(\hx,\cdot), \vy_{t}, \hK, \frac{1}{\ell}, \frac{\sqrt{\kappa}-1}{\sqrt{\kappa}+1}\big)$
% \begin{align*}
% \hy=\AGD\left(-f(\hx,\cdot), \vy_{t}, \hK, \frac{1}{\ell}, \frac{\sqrt{\kappa}-1}{\sqrt{\kappa}+1}\right),
% \end{align*}
where $t$ corresponds to the last iteration of Algorithm~\ref{alg:MCN} such that $\hx=\vx_t+\vs_t^*$ and $\hK = \sqrt{\kappa}\log\big(\min\big\{\frac{\alpha}{2\ell},~ \frac{\beta}{8\kappa^2\rho}\big\} \big/ \big(\sqrt{\kappa+1}\big(\teps + \frac{\kappa}{2^{2.25}}\sqrt{\frac{\eps}{M}}\,\big)\big)\big).$
% \begin{align}
% \hK = \sqrt{\kappa}\log\left(\frac{\min\left\{\frac{\alpha}{2\ell},~ \frac{\beta}{8\kappa^2\rho}\right\}}{\sqrt{\kappa+1}} \Bigg/ \left(\teps + \frac{\kappa}{2^{2.25}}\sqrt{\frac{\eps}{M}}\right)\right).
% \end{align}
Then there exists a local minimax point $(\vx^*, \vy^*)$ of $f(\vx,\vy)$ such that
$\norm{\vx^*-\hx}^2+\norm{\vy^*-\hy}^2\leq\gamma^2$.
\end{cor}

\section{Inexact Minimax Cubic Newton Algorithm} \label{sec:imcn}

In this section, we proposed an efficient algorithm called inexact minimax cubic Newton (IMCN), which avoids any operation related to the second-order oracle and only requires $\tilde\fO\left(\kappa^{1.5}\ell\eps^{-2}\right)$ Hessian-vector product calls and $\tilde\fO\left(\kappa^{2}\sqrt{\rho}\eps^{-1.5}\right)$ gradient calls in total to find $\left(\eps,\kappa^{1.5}\sqrt{\rho\eps}\,\right)$-SSP of $P(\vx)$. 
Since the Hessian-vector products can be computed as fast as gradients~\cite{pearlmutter1994fast,Schraudolph2002Fast}, IMCN is much more efficient than MCN in high-dimensional cases.

\begin{wrapfigure}[12]{R}{0.45\textwidth}
\vspace{-0.4cm}
\begin{minipage}{0.45\textwidth}
\begin{algorithm}[H]
\caption{Final-Cubic-Solver}\label{alg:cubic-final}
\begin{algorithmic}[1]
    \STATE \textbf{Input:} $\vg$, $\mH$, $\eps$  \\[0.1cm]
    \STATE $\vs_0=\vzero$, $\vg_0=\vg$, $\eta=1/(20L)$ \\[0.1cm]
    \STATE \textbf{for} $t=0,\dots$ \textbf{do} \\[0.1cm]
    \STATE~~ \textbf{if} $\norm{\vg_t}\leq\eps/2$ \textbf{then break}  \\[0.1cm]
    \STATE~~ $\vs_{t+1}= \vs_t - \eta\vg_t$ \\[0.1cm]
    \STATE~~ $\vg_{t+1} = \vg + \mH\vs_{t+1} + \frac{M}{2}\norm{\vs_{t+1}}\vs_{t+1}$
    \\[0.1cm]
    \STATE\textbf{end for} \\ [0.1cm]
    \STATE\textbf{Output:} $\vs_t$
\end{algorithmic}
\end{algorithm}
\end{minipage}
\end{wrapfigure}

We present the details of IMCN in Algorithm~\ref{alg:IMCN}. Unlike MCN which solves problem~(\ref{prob:sub-cubic}) exactly, IMCN uses gradient-based cubic sub-problem solver (Algorithm~\ref{alg:cubic-GD})~\cite{carmon2019gradient} to compute 
\begin{align*}
\vs_t \approx \argmin\limits_{\vs\in\BR^{d_x}} m_t(\vs)\triangleq\vg_t^\top \vs + \frac{1}{2}\vs^\top \mH_t \vs + \frac{M}{6}\norm{\vs}^3.
\end{align*}
%$\vs_t \approx \argmin\limits_{\vs\in\BR^{d_x}} m_t(\vs)\triangleq\vg_t^\top \vs + \frac{1}{2}\vs^\top \mH_t \vs + \frac{M}{6}\norm{\vs}^3$.
% \begin{align*} 
% \mbox{\small$\displaystyle{
% \vs_t \approx \argmin_{\vs\in\BR^{d_x}}\left(m_t(\vs)\triangleq\vg_t^\top \vs + \frac{1}{2}\vs^\top \mH_t \vs + \frac{M}{6}\norm{\vs}^3\right).}$}
% \end{align*}
If the condition in line 10 of Algorithm \ref{alg:IMCN} holds, the point $\vx_t+\vs_t^*$ should be a desired approximate second-order stationary point. Due to $\vs_t^*$ is hard to obtain, we introduce additional gradient descent steps (Algorithm~\ref{alg:cubic-final}) to approximate it by $\hs$ and use $\vx_t+\hs$ as the final output.

The design and the convergence analysis of IMCN is more challenging than existing inexact cubic Newton algorithms for minimization problems~\cite{kohler2017sub,tripuraneni2017stochastic} since the Hessian estimator in IMCN has a more complicated structure. 
To address this issue, we approximate the Hessian $\nabla^2 P(\vx_t)$ by $\mH_t$ as 
\begin{equation}\label{eq:Ht-poly}
    \mH_t = \nabla_{xx}^2 f(\vx_t,\vy_t) + \nabla_{xy}^2 f(\vx_t,\vy_t)\mC_t\nabla_{yx}^2 f(\vx_t,\vy_t),
\end{equation}
where $\mC_t=\frac{c_0}{4\ell}\mI+\frac{1}{2\ell}\sum_{k=1}^{K'}c_k\mT_k(\mZ_t)$ and $\mZ_t=\frac{4\ell}{\ell-\mu}\big(-\frac{1}{2\ell}\nabla_{yy}^2 f(\vx_t,\vy_t)-\frac{\ell+\mu}{4\ell}\mI\big).$
% \begin{align*}
% \mZ_t=\frac{4\ell}{\ell-\mu}\left(-\frac{1}{2\ell}\nabla_{yy}^2 f(\vx_t,\vy_t)-\frac{\ell+\mu}{4\ell}\mI\right).
% \end{align*}
Here $\mT_k(\cdot)$ is the matrix Chebyshev polynomials (shown in Section~\ref{sec:Chebyshev}) leading to $\mC_t\approx-\left(\nabla_{yy}^2(\vx_t,\vy_t)\right)^{-1}$.
% \begin{align*}
%     \mC_t\approx-\left(\nabla_{yy}^2(\vx_t,\vy_t)\right)^{-1}.
% \end{align*}
Note that we never construct matrix $\mH_t$ explicitly in implementation because all operations related to $\mH_t$ can be reduced to compute Hessian-vector products, which avoid any second-order oracle calls or matrix factorization/inversion.

Then we provide the convergence analysis for the IMCN algorithm. 
Throughout this section, we let $L=2\kappa\ell$ and $M=4\sqrt{2}\kappa^3\rho$ be the Lipschitz continuous constants of $\nabla P(x)$ and $\nabla^2 P(x)$. 
We suppose $\eps\leq L^2/M$, otherwise, the second-order condition $\nabla^2 P(\vx)\succeq-\sqrt{M\eps}\,\mI$ always holds and we only need to use gradient methods~\cite{lin2019gradient,lin2020near} to find first-order stationary point.

\subsection{Approximating Hessian by Matrix Chebyshev Polynomials}\label{sec:Chebyshev}

We first show the error bound of matrix inverse approximation via matrix Chebyshev polynomials.

\begin{lem}\label{lem:Chebyshev}
Suppose symmetric matrix $\mX\in\BR^{d\times d}$ satisfies $\mu' \mI \preceq  \mX \preceq \ell' \mI$ with $0<\mu'\leq\ell'<1$, then we have $\big\|\mX^{-1} - \big(\frac{c_0}{2}\mI+\sum_{k=1}^{K'}c_k\mT_k(\mZ')\big)\big\|_2 
\leq  \frac{\sqrt{\ell' /\mu'}-1}{\sqrt{\ell' \mu'}}\big(1-2/(\sqrt{\ell' /\mu'}\,+1)\big)^{K'}$
% {\small\begin{align*}
% \bigg\|\mX^{-1} - \Big(\frac{c_0}{2}\mI+\sum_{k=1}^{K'}c_k\mT_k(\mZ')\Big)\bigg\|_2 
% \leq  \frac{\sqrt{\ell' /\mu'}-1}{\sqrt{\ell' \mu'}}\bigg(1-\frac{2}{\sqrt{\ell' /\mu'}+1}\bigg)^{K'} , 
% \end{align*}}\\
where $\mZ'=\frac{2}{\ell' -\mu'}\big(\mX-\frac{\ell' +\mu'}{2}\mI\big)$, $c_k = \frac{2}{\sqrt{\ell' \mu'}}\Big(\frac{\sqrt{{\mu'/\ell'}}-1}{\sqrt{{\mu'/\ell'}}+1}\Big)^k$, and $\mT_k(\cdot)$ are matrix Chebyshev polynomials with $T_0(\mZ')=\mI$, $T_1(\mZ')=\mZ'$ and $\mT_k(\mZ')=2\mZ'\mT_{k-1}(\mZ') - \mT_{k-2}(\mZ')$ for $k \geq 2$.
% {\small\begin{align*}
% \displaystyle{c_k = \frac{2}{\sqrt{\ell' \mu'}}\left(\frac{\sqrt{{\mu'/\ell'}}-1}{\sqrt{{\mu'/\ell'}}+1}\right)^k}
% \quad\text{and}\quad \mT_k(\mZ')=
% \begin{cases}
% \mI, & k=0 \\
% \mZ', & k=1 \\
% 2\mZ'\mT_{k-1}(\mZ') - \mT_{k-2}(\mZ'), & k \geq 2 
% \end{cases}
% \end{align*}}

% \begin{align}\label{dfn:Tk}
% \mT_k(\mZ')=
% \begin{cases}
% \mI, & k=0 \\
% \mZ', & k=1 \\
% 2\mZ'\mT_{k-1}(\mZ') - \mT_{k-2}(\mZ'), & k \geq 2 
% \end{cases}
% \end{align}
% and
% $\displaystyle{c_k = \frac{2}{\sqrt{\ell' \mu'}}\left(\frac{\sqrt{{\mu'/\ell'}}-1}{\sqrt{{\mu'/\ell'}}+1}\right)^k}$.
\end{lem}

\noindent Based on Lemma~\ref{lem:Chebyshev}, we can bound the approximation error of the Hessian estimator $\mH_t$ as follows.
\begin{lem}\label{lem:inverse-approx}
Using the notation of Algorithm~\ref{alg:IMCN}, we have
\begin{align*}
  \norm{\nabla_{xx}^2P(\vx_t)-\mH_t}  
\leq  3\rho\kappa^2\sqrt{\kappa+1}\Big(1-\frac{1}{\sqrt{\kappa}}\Big)^{K_t/2}\norm{\vy_{t-1}-\vy^*(\vx_t)} 
+ \kappa\ell\Big(1-\frac{2}{\sqrt{\kappa}+1}\Big)^{K'}.
\end{align*}
\end{lem}
Lemma~\ref{lem:inverse-approx} means using AGD with $K_t=\fO\left(\sqrt{\kappa}\log(\kappa\rho/\eps_H)\right)=\tilde\fO\left(\sqrt{\kappa}\,\right)$\footnote{Rigorously speaking, the term $\log(\norm{\vy_{t-1}-\vy^*(\vx_t)})$ also should be considered into the total complexity, which will be discussed in later sections.} and the number of terms for Chebyshev polynomials with $K'=\fO\left(\sqrt{\kappa}\log(\kappa\ell/\eps_H)\right)=\tilde\fO\left(\sqrt{\kappa}\,\right)$ could achieve $\mH_t$ with $\norm{\nabla^2 P(\vx_t)-\mH_t}\leq\eps_H$ for any $\eps_H>0$.

In the implementation of IMCN, all operations related to $\mH_t$ can be viewed as computing Hessian-vector products. Actually, we can obtain $\mH_t\vu'$ with $\fO(K')=\tilde\fO(\sqrt{\kappa})$ Hessian-vector calls for any~$\vu'\in\BR^{d_x}$, which avoids $\fO\left(d_x^2+d_y^2\right)$ space to keep Hessian matrices. The detailed implementation is presented in Appendix~\ref{apsec:imple}.

\subsection{Complexity Analysis for IMCN}
The IMCN method calls a gradient-based sub-problem solver (Algorithm~\ref{alg:cubic-GD}) to optimize the following cubic regularized problem \cite{tripuraneni2017stochastic,carmon2019gradient} in each iteration
\begin{align}\label{prob:cubic-sub}
    \min_{\vs\in\BR^{d_x}}~\tm_t(\vs)\triangleq \vg_t^\top \vs + \frac{1}{2}\vs^\top \mH_t \vs + \frac{M}{6}\norm{\vs}^3.
\end{align}
It requires at most $\fK(\eps,\delta')=\fO\big(\frac{L}{\sqrt{M\eps}}\big(\log\big(\frac{\sqrt{d_x}}{\delta'}\big)+\log\big(\frac{L+\sqrt{\rho\eps}}{\sqrt{M\eps}}\big)\big)\big)$
% \begin{align}\label{eq:sub-iteration-number}
% \mbox{\small$\displaystyle{
% \fK(\eps,\delta')=\fO\left(\frac{L}{\sqrt{M\eps}}\left(\log\left(\frac{\sqrt{d_x}}{\delta'}\right)+\log\left(\frac{L+\sqrt{\rho\eps}}{\sqrt{M\eps}}\right)\right)\right)}$}
% \end{align}
number of iterations to achieve an approximate solution
with enough accuracy. The detailed analysis for the complexity of Algorithm~\ref{alg:cubic-GD} is deferred to appendix \ref{apsec:imcn}.

Then we bound the total number of iterations for Algorithm~\ref{alg:IMCN}.

\begin{thm}\label{thm:IMCN-T}
Running Algorithm~\ref{alg:IMCN} with $C_\sigma=1/4$, $\delta'=\delta/T$, $T=\lceil626(P(\vx_0) - P^*)\sqrt{M}\eps^{-1.5}\rceil$ and sufficient large $K_t$ and $K'$ such that
Assumption~\ref{asm:error-g-h} holds with $C_g=1/240$ and $C_H=1/200$.
Then the condition $\Delta_t\geq-\frac{1}{128}\sqrt{\eps^3/M}$ in line~11 must hold in no more than $T=\fO\left(\kappa^{1.5}\sqrt{\rho}\eps^{-1.5}\right)$ iterations; and the output $\hx$ is an $(\eps,2\kappa^{1.5}\sqrt{\rho\eps}\,)$-SSP with probability $1-\delta$.
\end{thm}

We also bound the number of gradient calls from AGD procedure in Algorithm~\ref{alg:IMCN} as follows. 
\begin{thm}\label{thm:IMCN-sum_K}
Under the setting of Theorem~\ref{thm:IMCN-T}, if we run Algorithm~\ref{alg:IMCN} with %$K'{=}\left\lceil\frac{\sqrt{\kappa}+1}{2}\log\Big(\frac{\kappa\ell}{2\min\big\{C_H\sqrt{M\eps}, \eps_HL\big\}}\Big)\right\rceil$, $K_0{=}\left\lceil2\sqrt{\kappa}\log\big(\frac{\sqrt{\kappa+1}}{\teps}\norm{\vy^*(\vx_0)}\big)\right\rceil$ and $K_t{=}\left\lceil2\sqrt{\kappa}\log\left(\frac{\sqrt{\kappa+1}}{\teps}\big(\teps + \kappa\norm{\vs_{t-1}}\big)\right)\right\rceil$ for $t\geq 1$
{\small\begin{align*}
K'{=}\bigg\lceil\frac{\sqrt{\kappa}{+}1}{2}\log\bigg(\frac{\kappa\ell}{2\min\big\{C_H\sqrt{M\eps}, \eps_HL\big\}}\bigg)\bigg\rceil ~\text{and}~
K_t {=} 
\begin{cases}
\left\lceil2\sqrt{\kappa}\log\left(\frac{\sqrt{\kappa+1}}{\teps}\norm{\vy^*(\vx_0)}\right)\right\rceil ,& \hspace{-0.1cm} t = 0 \\[0.2cm]
\left\lceil2\sqrt{\kappa}\log\left(\frac{\sqrt{\kappa+1}}{\teps}\Big(\teps + \kappa\norm{\vs_{t-1}}\Big)\right)\right\rceil,  & \hspace{-0.1cm} t\geq 1
\end{cases} 
\end{align*}}\\[-0.05cm]
where 
$\teps=\min\big\{C_g\eps/\ell, \min\big\{C_H\sqrt{M\eps}, \eps_HL\big\}/(6\rho\kappa^2)\big\}$ and $\vs_0 = \vzero$,
then it holds that
$\norm{\nabla P(\vx_t)-\vg_t} \leq  C_g\eps$, 
$\norm{\nabla^2 P(\vx_t)-\mH_t} \leq \min\big\{C_H\sqrt{M\eps}, \eps_HL\big\}$
and
{\small\begin{align*}
\sum_{t=0}^{T} K_t 
\leq & T {+} 1 {+} \frac{2\sqrt{\kappa}T}{3}\bigg[\frac{3}{T}\log\Big(\frac{\sqrt{\kappa+1}}{\teps}\norm{\vy^*(\vx_0)}\Big)
{+}\log\Big(8(\kappa+1)^{1.5} + \frac{8\kappa^3(\kappa+1)^{1.5}}{T\teps^3} \sum_{t=1}^{T}\norm{\vs_{t-1}}^3\Big)\bigg].
\end{align*}}
\end{thm}
Note that the value of $K'=\tilde\fO(\sqrt{\kappa})$ corresponds to the number of Hessian vector calls for each iteration of cubic sub-problem solver (Algorithm~\ref{alg:cubic-GD} and \ref{alg:cubic-final}). 
Combining Theorem~\ref{thm:IMCN-T}, Theorem \ref{thm:IMCN-sum_K} and the value of $\fK(\eps,\delta')$,
%equation (\ref{eq:sub-iteration-number}), 
we obtain the main result for Algorithm \ref{alg:IMCN} as follows.

\begin{cor}\label{cor:complexity2}
Under Assumption~\ref{asm:g-smooth}-\ref{asm:lower}, if we run Algorithm~\ref{alg:IMCN} with 
$C_\sigma=1/4$, $\delta'=\delta/T$, 
{\small\begin{align*}
K'{=}\left\lceil\frac{\sqrt{\kappa}{+}1}{2}\log\Bigg(\frac{\kappa\ell}{2\min\big\{C_H\sqrt{M\eps}, \frac{3L}{25}\big\}}\Bigg)\right\rceil ~\text{and}~
K_t {=} 
\begin{cases}
\left\lceil2\sqrt{\kappa}\log\left(\frac{\sqrt{\kappa+1}}{\teps}\norm{\vy^*(\vx_0)}\right)\right\rceil ,& \hspace{-0.1cm} t = 0 \\[0.2cm]
\left\lceil2\sqrt{\kappa}\log\left(\frac{\sqrt{\kappa+1}}{\teps}\big(\teps + \kappa\norm{\vs_{t-1}}\big)\right)\right\rceil,  & \hspace{-0.1cm} t\geq 1
\end{cases} 
\end{align*}}\\
where $T=\big\lceil626(P(\vx_0) - P^*)\sqrt{M}\eps'^{-1.5}\big\rceil$, $C_g=1/240$, $C_H=1/200$, $\vs_0=\vzero$, $\teps'=\eps/4$ and $\teps=\min\big\{C_g\eps/(4\ell), \min\big\{C_H\sqrt{M\eps}, \frac{3L}{25}\big\}\big/(12\rho\kappa^2)\big\}$,
% \begin{align*}
% \teps=\min\left\{\frac{C_g\eps}{4\ell}, \frac{\min\big\{C_H\sqrt{M\eps}, \frac{3L}{25}\big\}}{12\rho\kappa^2}\right\},
% \end{align*}
then the output $\hx$ is an $\big(\eps,\kappa^{1.5}\sqrt{\rho\eps}\,\big)$-second-order stationary point of $P(\vx)$ with probability $1-\delta$ and the number of gradient calls is at most $\tilde\fO\left(\kappa^2\sqrt{\rho}\eps^{-1.5}\right)$. The number of Hessian-vector product calls is at most $\fO\left(\kappa^{1.5}\ell\eps^{-2}\right)$.
\end{cor}

Following the analysis of Corollary~\ref{cor:local-minimax}, we can also achieve an approximate local minimax point by running AGD based on the $\hx$ obtained from Algorithm \ref{alg:IMCN}.

\section{Experiments} \label{sec:exp}
In this section, we conduct empirical studies for our methods against the classical GDA algorithm \cite{lin2019gradient} on both synthetic problem and real-world application.

\subsection{Synthetic Minimax Problem}

We construct the following nonconvex-strongly-concave minimax problem:
\begin{align}\label{prob:w}
\min_{\vx\in\BR^3} \max_{\vy\in\BR^2} f(\vx,\vy) =
w(x_3) - \frac{y_1^2}{40} + x_1 y_1 - \frac{5y_2^2}{2} + x_2y_2,
\end{align}
where $\vx=[x_1, x_2, x_3]^\top$, $\vy=[y_1, y_2]^\top$ and $w(\cdot)$ is the W-shaped scalar function~\cite{tripuraneni2017stochastic} whose exact form is shown in Appendix \ref{apsec:syn}. It is easy to verify that the problem has an strict saddle point at $(\vx_0,\vy_0)=([0,0,0]^\top, [0,0]^\top)$. 

% \begin{wrapfigure}[9]{R}{0.5\textwidth}
% %\begin{figure}[H] 
% \vspace{-0.4cm}
% \centering
% \begin{tabular}{cc}
% \includegraphics[clip,trim=10 0 0 0,scale=0.2]{figure/syn1_1.pdf} & \hspace{-0.5cm}
% \includegraphics[clip,trim=10 0 0 0,scale=0.2]{figure/syn2_1.pdf} \\
% \small(a) Initial point $(\vx_1,\vy_1)$ &  \hspace{-0.3cm} \small (b) Initial point $(\vx_2,\vy_2)$  \\[-0.1cm]
% \end{tabular}
% \caption{Results for the synthetic problem.}\label{fig:syn}
% \vskip-0.3cm
% %\end{figure}
% \end{wrapfigure}

\begin{figure}[ht] 
\centering 
\begin{tabular}{ccc}
    \hspace{-0.4cm}
    \includegraphics[scale=0.25]{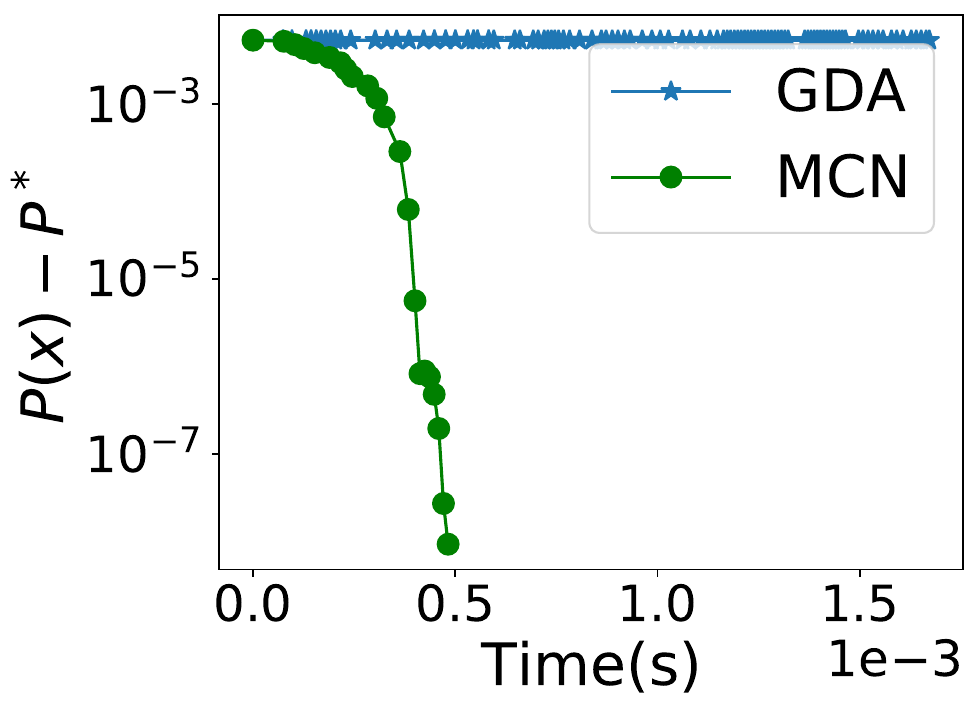} &
    \includegraphics[scale=0.25]{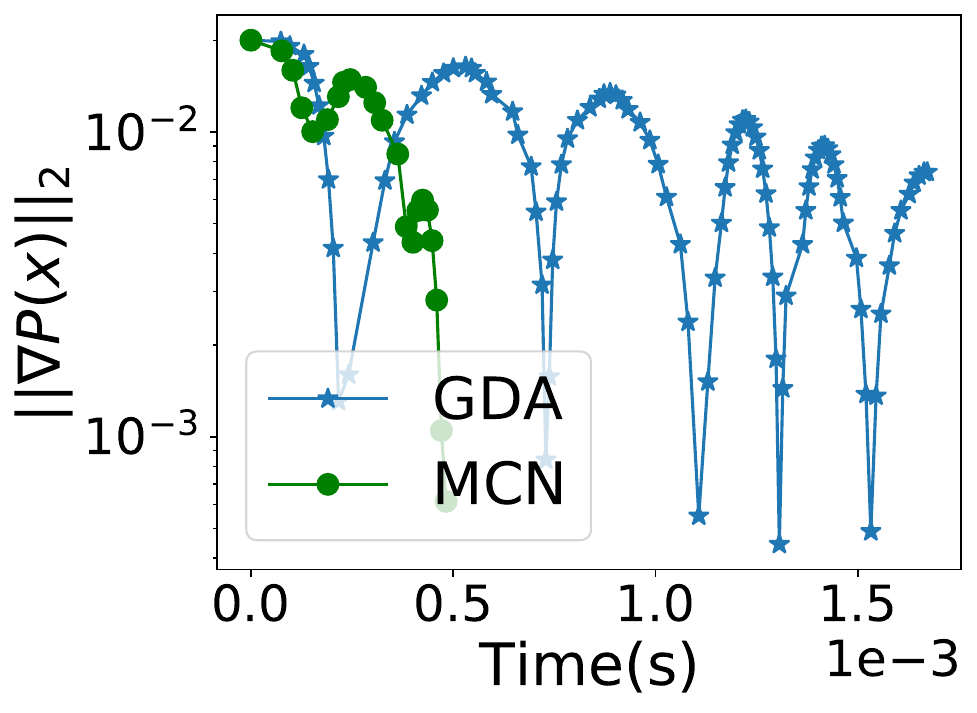} & 
    \includegraphics[scale=0.25]{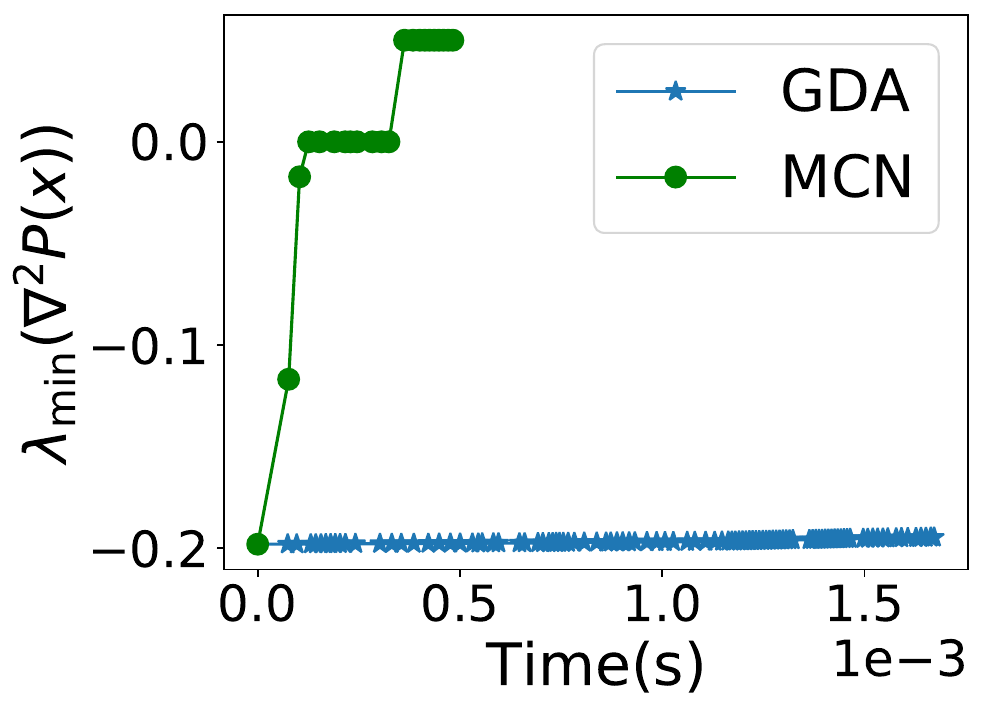}\\
    \small (a) Initial point $(\vx_1,\vy_1)$ & \small  (b) Initial point $(\vx_1,\vy_1)$ & \small  (c) Initial point $(\vx_1,\vy_1)$ \\
    \hspace{-0.4cm}
    \includegraphics[scale=0.25]{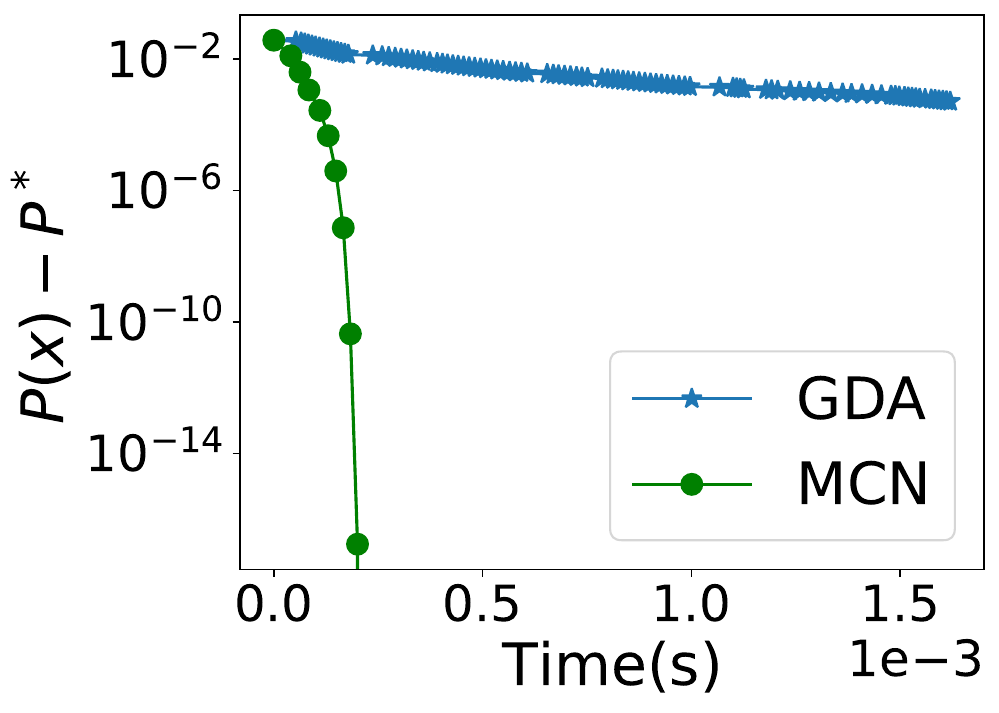} & 
    \includegraphics[scale=0.25]{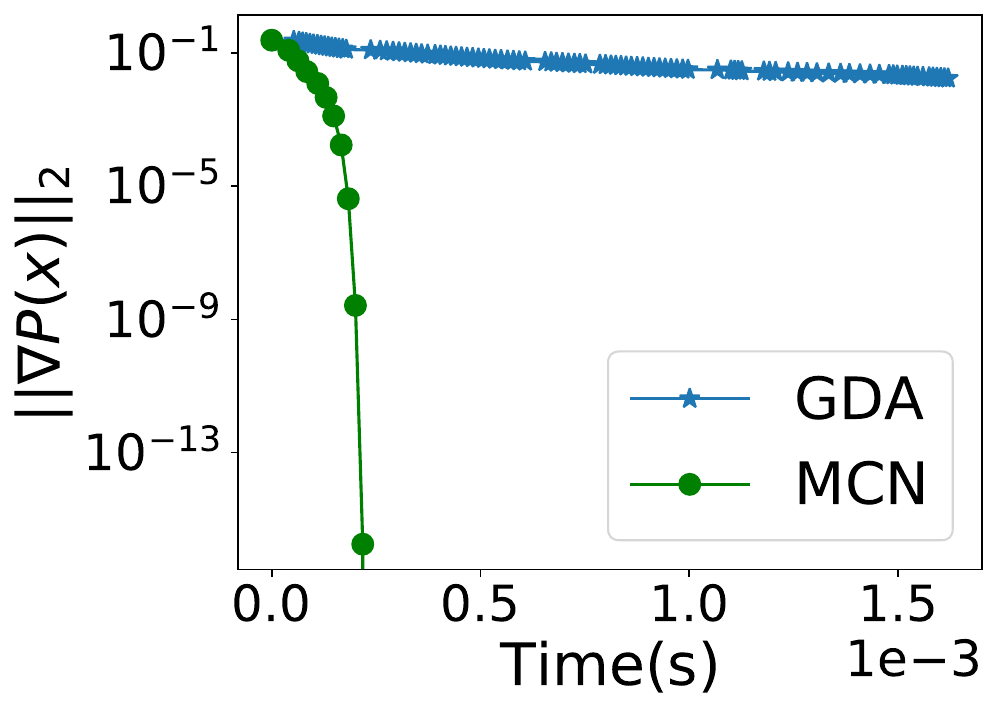} &
    \includegraphics[scale=0.25]{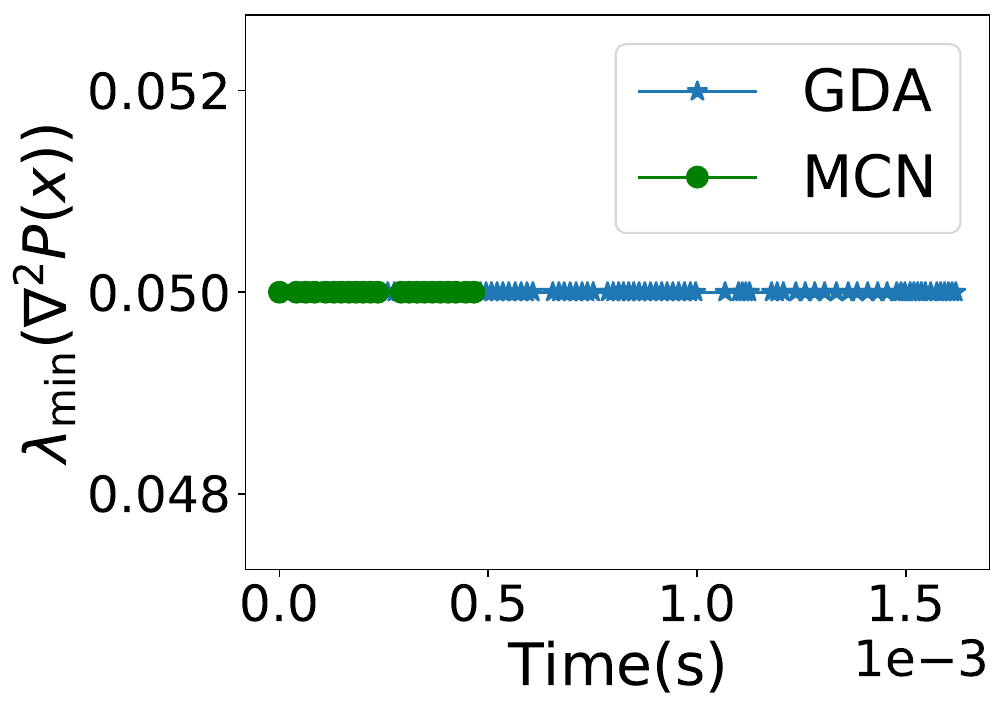}\\
    \small  (d) Initial point $(\vx_2,\vy_2)$ & \small  (e) Initial point $(\vx_2,\vy_2)$ & \small  (f) Initial point $(\vx_2,\vy_2)$  \\[-0.1cm]
\end{tabular}
\caption{We present comparisons of error, gradient norms and Hessian minimum eigenvalues on the synthetic problem. Figure (a), (b) and (c) shows the results with initial point $(\vx_1,\vy_1)$.  Figure (d), (e) and (f) shows the results with initial point $(\vx_2,\vy_2)$.}\label{fig:syn_gh}
\end{figure}

We conduct experiments on problem (\ref{prob:w}) with two different initial points
\begin{align*}
(\vx_1,\vy_1)=\left([10^{-3}, 10^{-3}, 10^{-3}]^\top,[0, 0]^\top\right) \quad\text{and}\quad
(\vx_2,\vy_2)=\left([0, 0, 1]^\top, [0, 0]^\top\right).     
\end{align*}
Notice that problem (\ref{prob:w}) has an strict saddle point at $(\vx_0,\vy_0)=([0,0,0]^\top, [0,0]^\top)$. The initial point $(\vx_1,\vy_1)$ is close to $(\vx_0,\vy_0)$ and $(\vx_2,\vy_2)$ is far from $(\vx_0,\vy_0)$. We compare the proposed algorithm MCN with GDA. The learning rate of GDA and AGD step in MCN is selected from $\left\{c\cdot 10^{-i}: c\in\{1,5\},i\in \{1,2,3\}\right\}$. For MCN method, we choose $M=10$. We compare the running time against $P(\vx)-P^*$, $\|\nabla P(\vx)\|_2$ and $\lambda_{\min}( \nabla^2 P(\vx))$ for two algorithms and plot the results in Figure \ref{fig:syn_gh}. From the curves corresponding to initial point $(\vx_2,\vy_2)$, we observe that both MCN and GDA converge to the minimum when the initial point is far from the strict saddle point, but MCN converges much faster than GDA. When the initial point is close to the strict saddle point, Figure 4(b) shows that the GDA algorithm gets stuck at the strict saddle point since its Hessian minimum eigenvalue are always negative. However, our MCN algorithm can reach the points which have positive Hessian minimum eigenvalues.

\subsection{Domain Adaptation}

The Domain-Adversarial Neural Network (DANN)~\cite{ganin2016domain} is a classic method to domain adaptation. Suppose the source domain dataset is $\fS=\{(\va_i^S,b_i^S)\}_{i=1}^{N_\fS}$ where $\va^\fS_i$ is the feature vector of the $i$-th sample and $b^\fS_i$ is the corresponding label. The target domain dataset $\fT=\{\va_i^\fT\}_{i=1}^{N_\fT}$ only contains features. Then DANN aims to solve the following nonconvex-strongly-concave minimax problem
\begin{align*}
    \min_{[\vx_1;\vx_2]\in\BR^{d_x}}\max_{\vy\in\BR^{d_y}} L_1(\vx_1,\vx_2) - \alpha\cdot L_2(\vx_1,\vy),
\end{align*}
where $L_1(\vx_1,\vx_2) = \frac{1}{N_\fS}\sum_{i=1}^{N_\fS}l(\vx_2;\Phi(\vx_1;\va_i^\fS),b_i^\fS)$ is the loss of supervised learning and
\begin{align*}
L_2(\vx_1,\vy)= \frac{1}{N_\fS}\sum_{i=1}^{N_\fS}D_\fS(h(\vy;\Phi(\vx_1;\va_i^\fS)))
-\frac{1}{N_\fT}\sum_{i=1}^{N_\fT}D_\fT(h(\vy;\Phi(\vx_1;\va_i^\fT)))+\lambda\|\vy\|^2
\end{align*} 
is the domain classification loss. Here $\Phi$ is a single-layer neural network of size $(28\times 28) \times 200$ with parameter $\vx_1$ and $l$ is a two-layer neural network of size $200 \times 20 \times 10$ with parameter $\vx_2$, followed by a cross entropy loss.% Thus, the dimension of x is around $160K$ and the dimension of y is $200$.
We choose the sigmoid function as the activation function for them. With the commonly used logistic loss for $L_2$, we let
$h(\vy;\vz)=1/(1+\exp(-\vy^\top\vz))$,
$D_S(z)=1-\log(z)$ and $D_T(z)=\log(1-z)$. 

\begin{figure}[t] 
\centering 
\begin{tabular}{cccc}
    \hspace{-0.2cm}
    \includegraphics[scale=0.2]{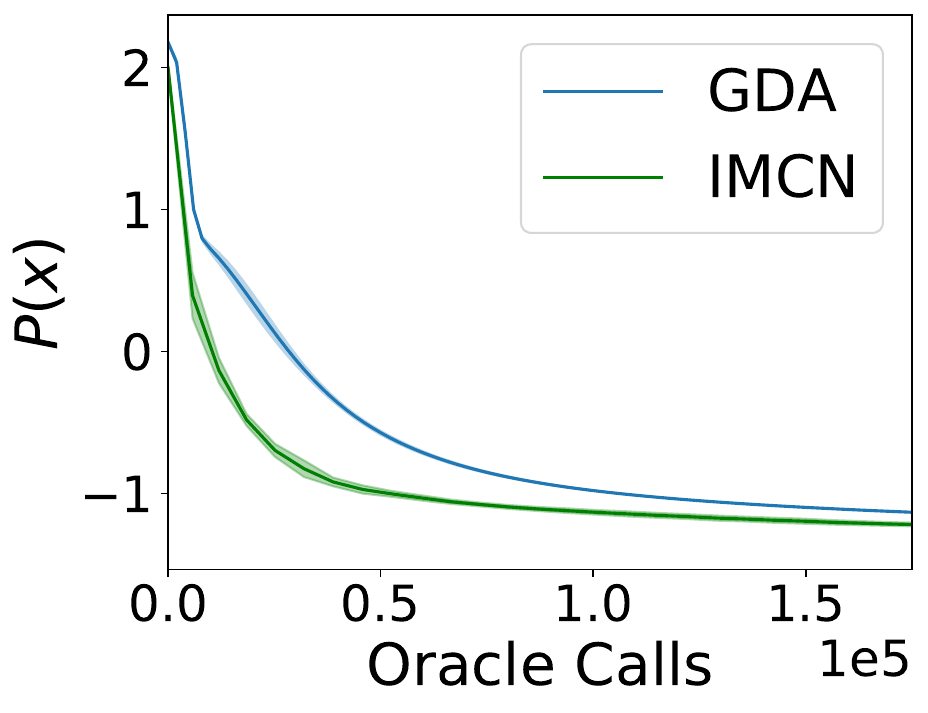} &
    \includegraphics[scale=0.2]{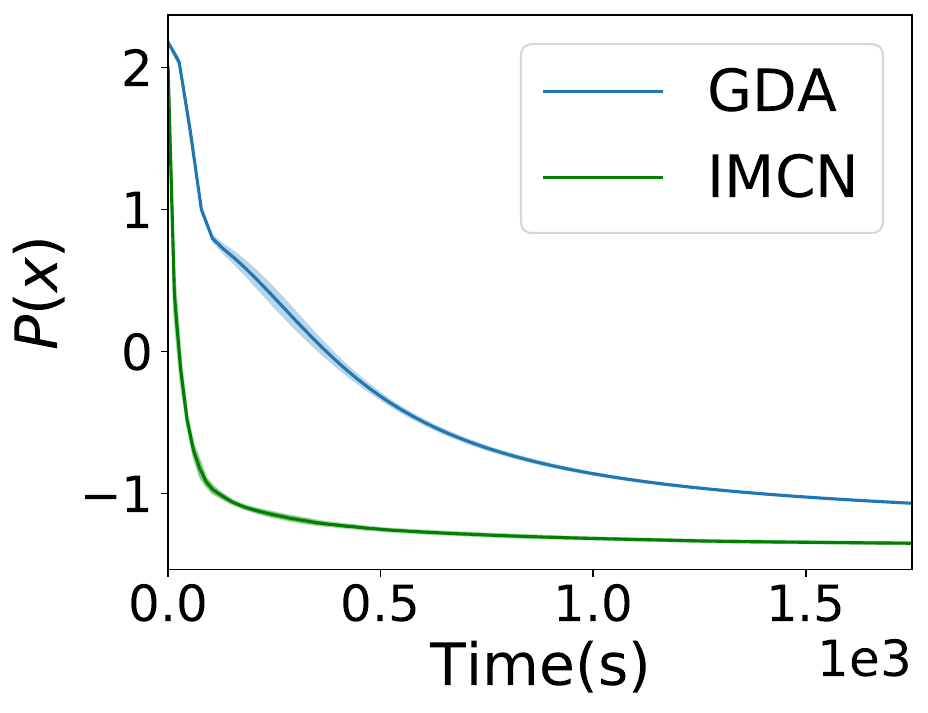} & 
    \includegraphics[scale=0.2]{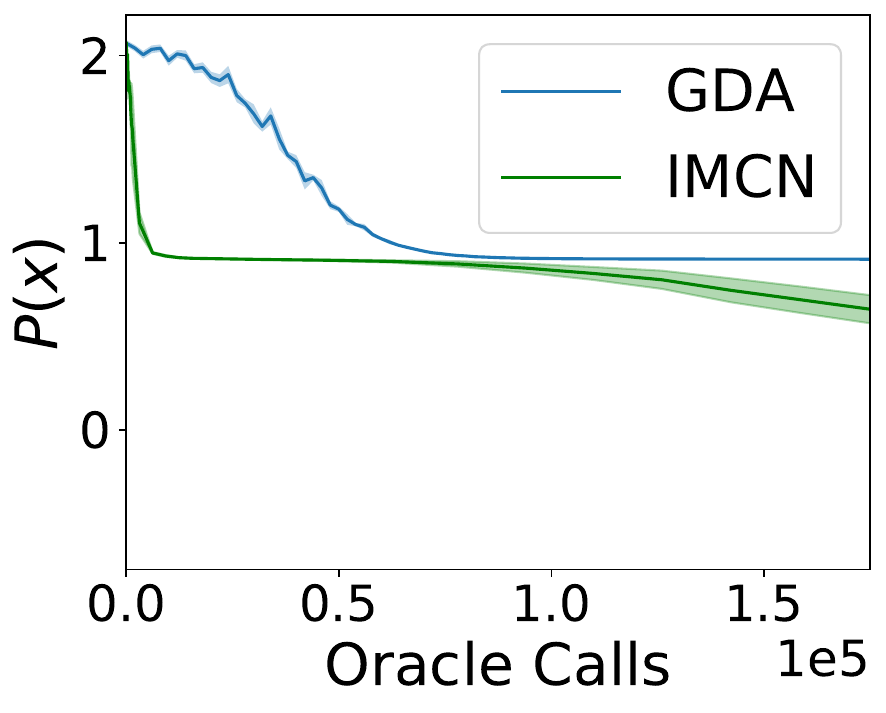} &
    \includegraphics[scale=0.2]{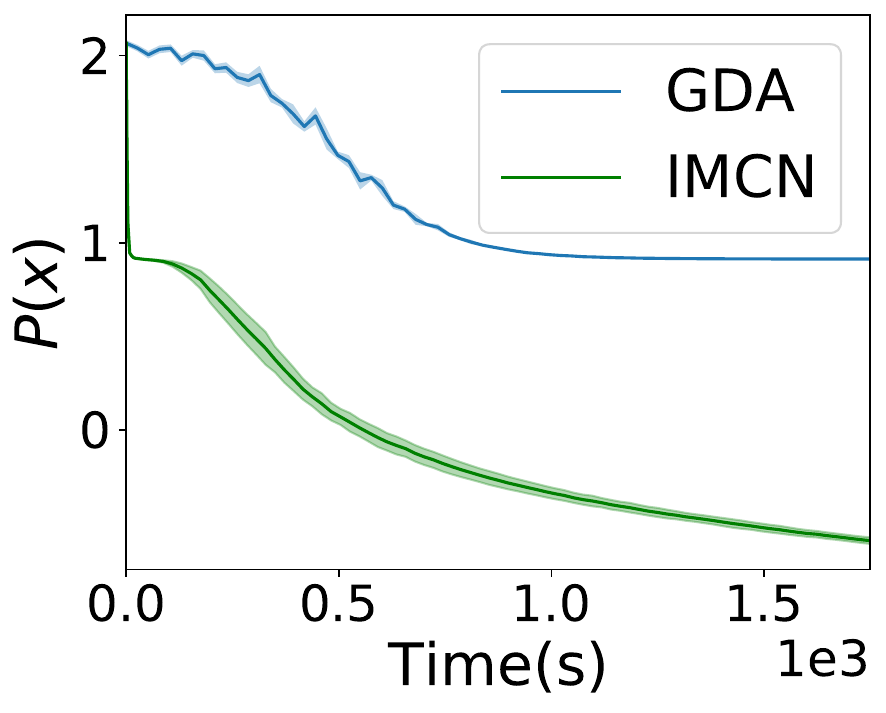}\\
    \small (a) MNIST to MNIST-m & \small  (b) MNIST to MNIST-m & \small  (c) MNIST-m to MNIST & \small  (d) MNIST-m to MNIST  \\[-0.1cm]
\end{tabular}
\caption{We present results on DANN model.  Figure (a) and (b) show the results of domain adaptation from MNIST to MNIST-m. Figure (c) and (d) show the results of domain adaptation from MNIST-m to MNIST.}\label{fig:domain}
\vskip-0.4cm
\end{figure}

Since the dimension of the minimax problem is quite large, we implement the IMCN algorithm instead of the MCN algorithm for efficiency. We compare IMCN with GDA on the domain adaptation problem between two different datasets: MNIST~\cite{lecun1998gradient} and MNIST-m~\cite{ganin2016domain}. Since we do not know the close form of $P(\vx)$, we estimate the value of $P(\vx)=\max_{\vy}f(\vx,\vy)$ by AGD procedure. More details about our experimental setup can be found in Appendix \ref{apsec:real}.

We compare IMCN and GDA algorithms via running time and oracle calls and show the results in Figure \ref{fig:domain}. We run the experiments for five times with different random initialization and report the average results. The oracle calls of GDA only contains gradient while the oracle calls of IMCN contains both gradient and Hessian-vector product. 
Based on Figure \ref{fig:domain}, we observe that IMCN significantly outperforms GDA in both time and oracle comparison. Notice that IMCN requires to call much more gradient/Hessian-vector oracles on y than the gradient oracles on $\vx$. These results verify our convergence analysis and show the advantage of proposed algorithm.

\section{Conclusions and Future Work} \label{sec:con}
In this paper, we study second-order optimization methods for nonconvex-strongly-concave minimax problems. 
We have proposed a novel algorithm so-called minimax cubic Newton (MCN) which could find an $(\eps,\kappa^{1.5}\sqrt{\rho\eps}\,)$-SSP of the primal function with $\fO(\kappa^{1.5}\sqrt{\rho}\eps^{-1.5})$ second-order oracle calls and $\tilde\fO(\kappa^{2}\sqrt{\rho}\eps^{-1.5})$ gradient oracle calls. 
We also provide an efficient algorithm for high dimensional problem, which avoids accessing second-order oracle and contains $\tilde\fO(\kappa^{1.5}\ell\eps^{-2})$ Hessian-vector oracle calls and $\tilde\fO(\kappa^{2}\sqrt{\rho}\eps^{-1.5})$ gradient oracle calls.
To best of our knowledge, this paper first achieves non-asymptotic convergence result for finding SSP of minimax problem without convex-concave assumption.

There are several interesting problems for future work:
(a) The proposed algorithms and analysis require the strongly convexity assumption on $\vy$. We would like to study how to find SSPs for general nonconvex-concave minimax problems.
(b) The upper complexity bounds of proposed algorithms look not optimal. 
It is possible to apply the acceleration techniques to establish more efficient algorithms for our task.
(c) The implementations of proposed IMCN still require accessing the Hessian-vector oracle. It is interesting to investigate how to find SSPs of our minimax problem by pure first-order algorithms. 
(d) This paper does not consider the specific structure of the objective function.
However, many machine learning models can be formulated as minimax problems where the objective functions have finite-sum or expectation form.
Designing efficient stochastic algorithms for such formulations is an interesting problem to the machine learning community.

\section*{Acknowledgements}
Luo Luo is supported by National Natural Science Foundation of China (No. 62206058) and Shanghai Sailing Program (22YF1402900). Cheng Chen is supported by Singapore Ministry of Education (AcRF) Tier 1 grant RG75/21.

\bibliography{reference}

\begin{thebibliography}{43}
\providecommand{\natexlab}[1]{#1}
\providecommand{\url}[1]{\texttt{#1}}
\expandafter\ifx\csname urlstyle\endcsname\relax
  \providecommand{\doi}[1]{doi: #1}\else
  \providecommand{\doi}{doi: \begingroup \urlstyle{rm}\Url}\fi

\bibitem[Agarwal et~al.(2017)Agarwal, Allen-Zhu, Bullins, Hazan, and
  Ma]{agarwal2017finding}
Naman Agarwal, Zeyuan Allen-Zhu, Brian Bullins, Elad Hazan, and Tengyu Ma.
\newblock Finding approximate local minima faster than gradient descent.
\newblock In \emph{STOC}, 2017.

\bibitem[Axelsson(1994)]{axelsson_1994}
Owe Axelsson.
\newblock \emph{Iterative Solution Methods}.
\newblock Cambridge University Press, 1994.

\bibitem[Bhojanapalli et~al.(2016)Bhojanapalli, Neyshabur, and
  Srebro]{bhojanapalli2016global}
Srinadh Bhojanapalli, Behnam Neyshabur, and Nati Srebro.
\newblock Global optimality of local search for low rank matrix recovery.
\newblock In \emph{Advances in Neural Information Processing Systems}, pages
  3873--3881, 2016.

\bibitem[Carmon and Duchi(2019)]{carmon2019gradient}
Yair Carmon and John Duchi.
\newblock Gradient descent finds the cubic-regularized nonconvex newton step.
\newblock \emph{SIAM Journal on Optimization}, 29\penalty0 (3):\penalty0
  2146--2178, 2019.

\bibitem[Cartis et~al.(2011{\natexlab{a}})Cartis, Gould, and
  Toint]{cartis2011aadaptive}
Coralia Cartis, Nicholas~I.M. Gould, and Philippe~L. Toint.
\newblock Adaptive cubic regularisation methods for unconstrained optimization.
  part {I}: motivation, convergence and numerical results.
\newblock \emph{Mathematical Programming}, 127\penalty0 (2):\penalty0 245--295,
  2011{\natexlab{a}}.

\bibitem[Cartis et~al.(2011{\natexlab{b}})Cartis, Gould, and
  Toint]{cartis2011badaptive}
Coralia Cartis, Nicholas~I.M. Gould, and Philippe~L. Toint.
\newblock Adaptive cubic regularisation methods for unconstrained optimization.
  part {II}: worst-case function-and derivative-evaluation complexity.
\newblock \emph{Mathematical programming}, 130\penalty0 (2):\penalty0 295--319,
  2011{\natexlab{b}}.

\bibitem[Chen et~al.(2021)Chen, Li, and Zhou]{chen2021escaping}
Ziyi Chen, Qunwei Li, and Yi~Zhou.
\newblock Escaping saddle points in nonconvex minimax optimization via
  cubic-regularized gradient descent-ascent.
\newblock \emph{arXiv preprint arXiv:2110.07098}, 2021.

\bibitem[Fiez et~al.(2021)Fiez, Ratliff, Mazumdar, Faulkner, and
  Narang]{fiez2021global}
Tanner Fiez, Lillian Ratliff, Eric Mazumdar, Evan Faulkner, and Adhyyan Narang.
\newblock Global convergence to local minmax equilibrium in classes of
  nonconvex zero-sum games.
\newblock In \emph{NeurIPS}, 2021.

\bibitem[Ganin et~al.(2016)Ganin, Ustinova, Ajakan, Germain, Larochelle,
  Laviolette, Marchand, and Lempitsky]{ganin2016domain}
Yaroslav Ganin, Evgeniya Ustinova, Hana Ajakan, Pascal Germain, Hugo
  Larochelle, Fran{\c{c}}ois Laviolette, Mario Marchand, and Victor Lempitsky.
\newblock Domain-adversarial training of neural networks.
\newblock \emph{Journal of Machine Learning Research}, 17\penalty0
  (1):\penalty0 2096--2030, 2016.

\bibitem[Ge et~al.(2015)Ge, Huang, Jin, and Yuan]{ge2015escaping}
Rong Ge, Furong Huang, Chi Jin, and Yang Yuan.
\newblock Escaping from saddle points-online stochastic gradient for tensor
  decomposition.
\newblock In \emph{COLT}, 2015.

\bibitem[Ge et~al.(2016)Ge, Lee, and Ma]{ge2016matrix}
Rong Ge, Jason~D. Lee, and Tengyu Ma.
\newblock Matrix completion has no spurious local minimum.
\newblock In \emph{NIPS}, 2016.

\bibitem[Guo et~al.(2020)Guo, Yuan, Yan, and Yang]{guo2020fast}
Zhishuai Guo, Zhuoning Yuan, Yan Yan, and Tianbao Yang.
\newblock Fast objective \& duality gap convergence for
  nonconvex-strongly-concave min-max problems.
\newblock \emph{arXiv preprint arXiv:2006.06889}, 2020.

\bibitem[Han et~al.(2021)Han, Xie, and Zhang]{han2021lower}
Yuze Han, Guangzeng Xie, and Zhihua Zhang.
\newblock Lower complexity bounds of finite-sum optimization problems: The
  results and construction.
\newblock \emph{arXiv preprint arXiv:2103.08280}, 2021.

\bibitem[Hanzely et~al.(2020)Hanzely, Doikov, Nesterov, and
  Richtarik]{hanzely2020stochastic}
Filip Hanzely, Nikita Doikov, Yurii Nesterov, and Peter Richtarik.
\newblock Stochastic subspace cubic newton method.
\newblock In \emph{ICML}, 2020.

\bibitem[Huang et~al.(2020)Huang, Gao, Pei, and Huang]{huang2020accelerated}
Feihu Huang, Shangqian Gao, Jian Pei, and Heng Huang.
\newblock Accelerated zeroth-order momentum methods from mini to minimax
  optimization.
\newblock \emph{arXiv preprint arXiv:2008.08170}, 2020.

\bibitem[Jin et~al.(2017)Jin, Ge, Netrapalli, Kakade, and
  Jordan]{jin2017escape}
Chi Jin, Rong Ge, Praneeth Netrapalli, Sham~M Kakade, and Michael~I Jordan.
\newblock How to escape saddle points efficiently.
\newblock In \emph{International Conference on Machine Learning}, pages
  1724--1732. PMLR, 2017.

\bibitem[Jin et~al.(2020)Jin, Netrapalli, and Jordan]{jin2019local}
Chi Jin, Praneeth Netrapalli, and Michael~I. Jordan.
\newblock What is local optimality in nonconvex-nonconcave minimax
  optimization?
\newblock In \emph{ICML}, 2020.

\bibitem[Kohler and Lucchi(2017)]{kohler2017sub}
Jonas~Moritz Kohler and Aurelien Lucchi.
\newblock Sub-sampled cubic regularization for non-convex optimization.
\newblock In \emph{ICML}, 2017.

\bibitem[LeCun et~al.(1998)LeCun, Bottou, Bengio, and
  Haffner]{lecun1998gradient}
Yann LeCun, L{\'e}on Bottou, Yoshua Bengio, and Patrick Haffner.
\newblock Gradient-based learning applied to document recognition.
\newblock \emph{Proceedings of the IEEE}, 86\penalty0 (11):\penalty0
  2278--2324, 1998.

\bibitem[Lin et~al.(2020{\natexlab{a}})Lin, Jin, and Jordan]{lin2019gradient}
Tianyi Lin, Chi Jin, and Michael~I. Jordan.
\newblock On gradient descent ascent for nonconvex-concave minimax problems.
\newblock In \emph{ICML}, 2020{\natexlab{a}}.

\bibitem[Lin et~al.(2020{\natexlab{b}})Lin, Jin, and Jordan]{lin2020near}
Tianyi Lin, Chi Jin, and Michael~I. Jordan.
\newblock Near-optimal algorithms for minimax optimization.
\newblock In \emph{COLT}, 2020{\natexlab{b}}.

\bibitem[Luo et~al.(2020)Luo, Ye, Huang, and Zhang]{luo2020stochastic}
Luo Luo, Haishan Ye, Zhichao Huang, and Tong Zhang.
\newblock Stochastic recursive gradient descent ascent for stochastic
  nonconvex-strongly-concave minimax problems.
\newblock In \emph{NeurIPS}, 2020.

\bibitem[Luo et~al.(2021)Luo, Xie, Zhang, and Zhang]{luo2021near}
Luo Luo, Guangzeng Xie, Tong Zhang, and Zhihua Zhang.
\newblock Near optimal stochastic algorithms for finite-sum unbalanced
  convex-concave minimax optimization.
\newblock \emph{arXiv preprint arXiv:2106.01761}, 2021.

\bibitem[Mazumdar et~al.(2019)Mazumdar, Jordan, and
  Sastry]{mazumdar2019finding}
Eric~V Mazumdar, Michael~I Jordan, and S~Shankar Sastry.
\newblock On finding local nash equilibria (and only local nash equilibria) in
  zero-sum games.
\newblock \emph{arXiv preprint arXiv:1901.00838}, 2019.

\bibitem[Nesterov(2018)]{nesterov2018lectures}
Yurii Nesterov.
\newblock \emph{Lectures on convex optimization}, volume 137.
\newblock Springer, 2018.

\bibitem[Nesterov and Polyak(2006)]{nesterov2006cubic}
Yurii Nesterov and Boris~T. Polyak.
\newblock Cubic regularization of newton method and its global performance.
\newblock \emph{Mathematical Programming}, 108\penalty0 (1):\penalty0 177--205,
  2006.

\bibitem[Nesterov(1983)]{nesterov1983method}
Yurii~E Nesterov.
\newblock A method for solving the convex programming problem with convergence
  rate $o(1/k^2)$.
\newblock In \emph{Dokl. akad. nauk Sssr}, volume 269, pages 543--547, 1983.

\bibitem[Pearlmutter(1994)]{pearlmutter1994fast}
Barak~A. Pearlmutter.
\newblock Fast exact multiplication by the {H}essian.
\newblock \emph{Neural computation}, 6\penalty0 (1):\penalty0 147--160, 1994.

\bibitem[Qiu et~al.(2020)Qiu, Yang, Wei, Ye, and Wang]{qiu2020single}
Shuang Qiu, Zhuoran Yang, Xiaohan Wei, Jieping Ye, and Zhaoran Wang.
\newblock Single-timescale stochastic nonconvex-concave optimization for smooth
  nonlinear {TD} learning.
\newblock \emph{arXiv preprint arXiv:2008.10103}, 2020.

\bibitem[Sanjabi et~al.(2018)Sanjabi, Ba, Razaviyayn, and
  Lee]{sanjabi2018convergence}
Maziar Sanjabi, Jimmy Ba, Meisam Razaviyayn, and Jason~D. Lee.
\newblock On the convergence and robustness of training {GAN}s with regularized
  optimal transport.
\newblock \emph{arXiv preprint arXiv:1802.08249}, 2018.

\bibitem[Schraudolph(2002)]{Schraudolph2002Fast}
Nicol~N. Schraudolph.
\newblock Fast curvature matrix-vector products for second-order gradient
  descent.
\newblock \emph{Neural Computation}, 14\penalty0 (7):\penalty0 1723, 2002.

\bibitem[Shapiro(1985)]{shapiro1985second}
Alexander Shapiro.
\newblock Second-order derivatives of extremal-value functions and optimality
  conditions for semi-infinite programs.
\newblock \emph{Mathematics of Operations Research}, 10\penalty0 (2):\penalty0
  207--219, 1985.

\bibitem[Sinha et~al.(2018)Sinha, Namkoong, and Duchi]{sinha2017certifying}
Aman Sinha, Hongseok Namkoong, and John Duchi.
\newblock Certifying some distributional robustness with principled adversarial
  training.
\newblock In \emph{ICLR}, 2018.

\bibitem[Sun et~al.(2016)Sun, Qu, and Wright]{sun2016complete}
Ju~Sun, Qing Qu, and John Wright.
\newblock Complete dictionary recovery over the sphere {I}: Overview and the
  geometric picture.
\newblock \emph{IEEE Transactions on Information Theory}, 63\penalty0
  (2):\penalty0 853--884, 2016.

\bibitem[Sun et~al.(2018)Sun, Qu, and Wright]{sun2018geometric}
Ju~Sun, Qing Qu, and John Wright.
\newblock A geometric analysis of phase retrieval.
\newblock \emph{Foundations of Computational Mathematics}, 18\penalty0
  (5):\penalty0 1131--1198, 2018.

\bibitem[Tripuraneni et~al.(2018)Tripuraneni, Stern, Jin, Regier, and
  Jordan]{tripuraneni2017stochastic}
Nilesh Tripuraneni, Mitchell Stern, Chi Jin, Jeffrey Regier, and Michael~I.
  Jordan.
\newblock Stochastic cubic regularization for fast nonconvex optimization.
\newblock In \emph{NeurIPS}, 2018.

\bibitem[Wang and Li(2020)]{wang2020improved}
Yuanhao Wang and Jian Li.
\newblock Improved algorithms for convex-concave minimax optimization.
\newblock In \emph{NeurIPS}, 2020.

\bibitem[Wang et~al.(2019)Wang, Zhang, and Ba]{wang2019solving}
Yuanhao Wang, Guodong Zhang, and Jimmy Ba.
\newblock On solving minimax optimization locally: A follow-the-ridge approach.
\newblock In \emph{ICLR}, 2019.

\bibitem[Xian et~al.(2021)Xian, Huang, Zhang, and Huang]{xian2021faster}
Wenhan Xian, Feihu Huang, Yanfu Zhang, and Heng Huang.
\newblock A faster decentralized algorithm for nonconvex minimax problems.
\newblock \emph{NeurIPS}, 2021.

\bibitem[Xu et~al.(2020)Xu, Wang, Liang, and Poor]{xu2020gradient}
Tengyu Xu, Zhe Wang, Yingbin Liang, and H~Vincent Poor.
\newblock Gradient free minimax optimization: Variance reduction and faster
  convergence.
\newblock \emph{arXiv preprint arXiv:2006.09361}, 2020.

\bibitem[Zhang et~al.(2020)Zhang, Wu, Poupart, and Yu]{zhang2020newton}
Guojun Zhang, Kaiwen Wu, Pascal Poupart, and Yaoliang Yu.
\newblock Newton-type methods for minimax optimization.
\newblock \emph{arXiv preprint arXiv:2006.14592}, 2020.

\bibitem[Zhang et~al.(2021)Zhang, Yang, Guzm{\'a}n, Kiyavash, and
  He]{zhang2021complexity}
Siqi Zhang, Junchi Yang, Crist{\'o}bal Guzm{\'a}n, Negar Kiyavash, and Niao He.
\newblock The complexity of nonconvex-strongly-concave minimax optimization.
\newblock \emph{arXiv preprint arXiv:2103.15888}, 2021.

\bibitem[Zhou et~al.(2019)Zhou, Xu, and Gu]{zhou2019stochastic}
Dongruo Zhou, Pan Xu, and Quanquan Gu.
\newblock Stochastic variance-reduced cubic regularization methods.
\newblock \emph{Journal of Machine Learning Research}, 20\penalty0
  (134):\penalty0 1--47, 2019.

\end{thebibliography}
\bibliographystyle{plainnat}

\begin{enumerate}

\item For all authors...
\begin{enumerate}
  \item Do the main claims made in the abstract and introduction accurately reflect the paper's contributions and scope?
    \answerYes{}
  \item Did you describe the limitations of your work?
    \answerYes{}
  \item Did you discuss any potential negative societal impacts of your work?
    \answerNA{}
  \item Have you read the ethics review guidelines and ensured that your paper conforms to them?
    \answerYes{}
\end{enumerate}

\item If you are including theoretical results...
\begin{enumerate}
  \item Did you state the full set of assumptions of all theoretical results?
    \answerYes{}
        \item Did you include complete proofs of all theoretical results?
    \answerYes{}
\end{enumerate}

\item If you ran experiments...
\begin{enumerate}
  \item Did you include the code, data, and instructions needed to reproduce the main experimental results (either in the supplemental material or as a URL)?
    \answerYes{}
  \item Did you specify all the training details (e.g., data splits, hyperparameters, how they were chosen)?
    \answerYes{}
        \item Did you report error bars (e.g., with respect to the random seed after running experiments multiple times)?
    \answerYes{}
        \item Did you include the total amount of compute and the type of resources used (e.g., type of GPUs, internal cluster, or cloud provider)?
    \answerYes{}
\end{enumerate}

\item If you are using existing assets (e.g., code, data, models) or curating/releasing new assets...
\begin{enumerate}
  \item If your work uses existing assets, did you cite the creators?
    \answerYes{}
  \item Did you mention the license of the assets?
    \answerNA{}
  \item Did you include any new assets either in the supplemental material or as a URL?
    \answerNA{}
  \item Did you discuss whether and how consent was obtained from people whose data you're using/curating?
    \answerNA{}
  \item Did you discuss whether the data you are using/curating contains personally identifiable information or offensive content?
    \answerNA{}
\end{enumerate}

\item If you used crowdsourcing or conducted research with human subjects...
\begin{enumerate}
  \item Did you include the full text of instructions given to participants and screenshots, if applicable?
    \answerNA{}
  \item Did you describe any potential participant risks, with links to Institutional Review Board (IRB) approvals, if applicable?
    \answerNA{}
  \item Did you include the estimated hourly wage paid to participants and the total amount spent on participant compensation?
    \answerNA{}
\end{enumerate}

\end{enumerate}

%%%%%%%%%%%%%%%%%%%%%%%%%%%%%%%%%%%%%%%%%%%%%%%%%%%%%%%%%%%%

\newpage
\appendix

The appendix is organized as follows. We first present some basic lemmas used in our proofs in Section~\ref{apsec:basic}. Then we provide the proofs for Section~\ref{sec:pre}, Section~\ref{section:MCN} and Section~\ref{sec:imcn} in Section~\ref{apsec:pre}, Section~\ref{apsec:mcn} and Section~\ref{apsec:imcn}, respectively. We compare the complexity of MCN and IMCN with first-order methods in Section~\ref{apsec:com}. We introduce the implementation details of IMCN in Section~\ref{apsec:imple} and provide the details of experiments in Section~\ref{apsec:exp}.

\section{Basic Lemmas} \label{apsec:basic}
In this section, we provide some basic lemmas.
\begin{lem}\label{lem:lip-another}
Under assumptions of Lemma~\ref{lem:P-Hessian}, then
$\nabla_{xy}^2 f(\vx, \cdot) \left(\nabla_{yy}^2 f(\vx, \cdot)\right)^{-1} \nabla_{yx}^2 f(\vx, \cdot)$ is $3\kappa^2\rho$-Lipschitz continuous for fixed $\vx$.
\end{lem}
\begin{proof}
For fixed $x\in\BR^{d_x}$ and any $y,y'\in\BR^{d_y}$, we have
\begin{align*}
& \norm{\left(\nabla^2_{yy}f(\vx,\vy)\right)^{-1}-\left(\nabla^2_{yy}f(\vx,\vy')\right)^{-1}} \\
=& \norm{\left(\nabla^2_{yy}f(\vx,\vy)\right)^{-1}\left(\nabla^2_{yy}f(\vx,\vy)-\nabla^2_{yy}f(\vx,\vy')\right)\left(\nabla^2_{yy}f(\vx,\vy')\right)^{-1}} \\
\leq & \frac{1}{\mu^2}\norm{\nabla^2_{yy}f(\vx,\vy)-\nabla^2_{yy}f(\vx,\vy')} 
\leq  \frac{\rho}{\mu^2} 
\end{align*}
and
\begin{align*}
\norm{\left(\nabla^2_{yy}f(\vx,\vy)\right)^{-1}} \leq \frac{1}{\mu}.
\end{align*}
Similar to the analysis of Lemma~\ref{lem:Ph-smooth}, 
we apply Lemma~\ref{lem:lip} on $\left(\nabla^2_{yy}f(\vx,\cdot\right)^{-1}$ and $\nabla^2_{yx}f(\vx,\cdot)$ which means
$\left(\nabla_{yy}^2 f(\vx,\cdot)\right)^{-1}\nabla^2_{yx}f(\vx,\cdot)$ is $(\kappa+1)\rho/\mu$-Lipschitz continuous and its norm is bounded by $\kappa$. Then again applying Lemma~\ref{lem:lip} on
$\nabla^2_{xy}f(\vx,\cdot)$ and
$\left(\nabla_{yy}^2f(\vx,\cdot)\right)^{-1}\nabla^2_{yx}f(\vx,\cdot)$ means the Lipschitz continuous coefficient of $\nabla_{xy}^2 f(\vx, \cdot) \left(\nabla_{yy}^2 f(\vx, \cdot)\right)^{-1} \nabla_{yx}^2 f(\vx, \cdot)$ is $(\kappa^2+2\kappa)\rho\leq 3\kappa^2\rho$.
\end{proof}

Then we introduce two classical lemmas for cubic Newton methods.
\begin{lem}[{\cite{nesterov2006cubic}}]\label{lem:h-smooth}
Under assumptions of Lemma \ref{lem:Ph-smooth}, we have
\begin{align}
\norm{\nabla P(\vx') - \nabla P(\vx) - \nabla^2 P(\vx)(\vx'-\vx)} \leq  \frac{M}{2}\norm{\vx-\vx'}^2 \label{ieq:h-smooth1}
\end{align}
and
\begin{align}
\left|P(\vx') - P(\vx) - \nabla P(\vx)^\top(\vx'-\vx) - \frac{1}{2}(\vx'-\vx)^\top\nabla^2 P(\vx) (\vx'-\vx)\right| \leq  \frac{M}{6}\norm{\vx-\vx'}^3 \label{ieq:h-smooth2}.
\end{align}
\end{lem}

\begin{lem}[{\cite{nesterov2006cubic}}]\label{lem:cubic-exact-solution}
For any $M'>0$, the solution $\vs^*$ of the following cubic regularized quadratic problem  
\begin{align*}
\vs^* = \argmin_{\vx\in\BR^{d_x}} \left(\vg^\top \vs + \frac{1}{2}\vs^\top \mH \vs + \frac{M'}{6}\norm{\vs}^3\right)
\end{align*}
satisfies
\begin{align}
\vg + \mH \vs^* + \frac{M'}{2}\norm{\vs^*}\vs^* = & \vzero, \label{ieq:cubic-sol-1}, \\
\mH + \frac{M'}{2}\norm{\vs^*}\mI \succeq & \vzero \label{ieq:cubic-sol-2}, \\
\vg^\top \vs^* + \dfrac{1}{2}(\vs^*)^\top \mH \vs^* + \dfrac{M'}{6}\norm{\vs^*}^3 \leq &  -\frac{M'}{12}\norm{\vs^*}^3 \label{ieq:cubic-sol-3}.
\end{align}
\end{lem}

\section{Proofs for Section \ref{sec:pre}} \label{apsec:pre}
In this section, we provide the proofs for lemmas in Section \ref{sec:pre}.

\subsection{The Proof of Lemma \ref{lem:P-Hessian}}
\begin{proof}
The implicit function theorem means $\vy^*(\cdot)$ is differentiable.
The optimality of $\vy^*$ means 
\begin{align*}
\nabla_y f(\vx, \vy^*(\vx)) = \vzero.
\end{align*}
Taking total derivative on above equation achieves
\begin{align}\label{eq:P-Hessian-0}
\nabla^2_{yx} f(\vx, \vy^*(\vx)) + \nabla^2_{yy} f(\vx, \vy^*(\vx))\nabla \vy^*(\vx)=\vzero.
\end{align}
Taking total derivative on the result of Lemma~\ref{lem:P-smooth}, we have
\begin{align}\label{eq:P-Hessian-1}
\nabla^2 P(\vx) = \nabla_{xx}^2 f(\vx,\vy^*(\vx)) + \nabla_{xy}^2 f(\vx,\vy^*(\vx))\nabla \vy^*(\vx).
\end{align}
The strongly concavity on $\vy$ implies $\nabla^2_{yy} f(\vx,\vy)$ is non-singular. Hence, we connect (\ref{eq:P-Hessian-0}) and (\ref{eq:P-Hessian-1}) to obtain the desired result.
\end{proof}

\subsection{The Proof of Lemma~\ref{lem:Ph-smooth}}

We first introduce the following lemma.
\begin{lem}\label{lem:lip}
Suppose operators $\fH_1:\fD\to\fF$ and $\fH_2:\fD\to\fF$ are $\alpha_1$-Lipschitz continuous and $\alpha_2$-Lipschitz continuous with respect to norms $\Norm{\cdot}_\fD$ and $\Norm{\cdot}_\fF$ defined on $\fD$ and $\fF$, that is, for any $z$ and $z'$ in $\fD$, we have
\begin{align*}
\Norm{\fH_1(\mZ) - \fH_1(\mZ')}_\fF \leq \alpha_1\Norm{\mZ-\mZ'}_\fD \quad\text{and}\quad
\Norm{\fH_2(\mZ) - \fH_2(\mZ')}_\fF \leq \alpha_2\Norm{\mZ-\mZ'}_\fD.
\end{align*}
We further suppose that there exist some $\beta_1, \beta_2\geq0$ such that $\Norm{\fH_1(\mZ)}_\fF\leq \beta_1$ and $\Norm{\fH_2(\mZ)}_\fF\leq \beta_2$ for any $z$ in $\fD$; and norm $\Norm{\cdot}_{\fF}$ is submultiplicative on $\fF$, that is, for any $\mZ$ and $\mZ'$ in $\fF$, we have $\Norm{\mZ\mZ'}_{\fF}\leq\Norm{\mZ}_{\fF}\Norm{\mZ'}_{\fF}$ if $\mZ\mZ'$ also in $\fF$. Then, the operator $\fH_1(\cdot)\fH_2(\cdot)$ is $(\alpha_1\beta_2+\alpha_2\beta_1)$-Lipschitz continuous, that is, for any $\mZ,\mZ'\in\fD$, we have
\begin{align*}
 \Norm{\fH_1(\mZ)\fH_2(\mZ) - \fH_1(\mZ')\fH_2(\mZ')}_\fF 
\leq  (\alpha_1\beta_2+\alpha_2\beta_1)\Norm{\mZ-\mZ'}_\fD.
\end{align*}
\end{lem}
\begin{proof}
For any $\mZ,\mZ'\in\fD$, we have
\begin{align*}
 & \Norm{\fH_1(\mZ)\fH_2(\mZ) - \fH_1(\mZ')\fH_2(\mZ')}_\fF \\
\leq & \Norm{\fH_1(\mZ)\fH_2(\mZ) - \fH_1(\mZ)\fH_2(\mZ')}_\fF + \Norm{\fH_1(\mZ)\fH_2(\mZ') - \fH_1(\mZ')\fH_2(\mZ')}_\fF \\
\leq & \Norm{\fH_1(\mZ)}_\fF \Norm{\fH_2(\mZ) - \fH_2(\mZ')}_\fF + \Norm{\fH_2(\mZ')}_\fF \Norm{\fH_1(\mZ) - \fH_1(\mZ')}_\fF \\
\leq & (\alpha_1\beta_2+\alpha_2\beta_1)\Norm{\mZ-\mZ'}_\fD.
\end{align*}
\end{proof}

\noindent Then we prove Lemma~\ref{lem:Ph-smooth} as follows.
\begin{proof}
Using Lemma~\ref{lem:P-Hessian} and triangle inequality, we have
{\begin{align}\label{ieq:P-lc-1}
\begin{split}
& \norm{\nabla^2 P(\vx) - \nabla^2 P(\vx')} \\
\leq &\norm{\nabla^2_{xx} f(\vx,\vy^*(\vx)) - \nabla^2_{xx} f(\vx',\vy^*(\vx'))} + \\
& \big\|\nabla^2_{xy} f(\vx,\vy^*(\vx))(\nabla^2_{yy} f(\vx,\vy^*(\vx)))^{-1}\nabla^2_{yx} f(\vx,\vy^*(\vx)) \\
&~~~ -\nabla^2_{xy} f(\vx',\vy^*(\vx')) (\nabla^2_{yy} f(\vx',\vy^*(\vx')))^{-1}\nabla^2_{yx} f(\vx',\vy^*(\vx'))\big\|. 
\end{split}
\end{align}}
Assumption~\ref{asm:h-smooth} means $\nabla_{xx}^2 f(\cdot,\cdot)$ is $\rho$-Lipschitz continuous, which implies for any $x$ and $x'$, we have
\begin{align*}
 & \norm{\nabla^2_{xx} f(\vx,\vy^*(\vx)) - \nabla^2_{xx} f(\vx',\vy^*(\vx'))}^2 \\
\leq & \rho^2\left(\norm{\vx-\vx'}^2+\norm{\vy^*(\vx) - \vy^*(\vx')}^2\right) \\
\leq & \rho^2\left(\norm{\vx-\vx'}^2+\kappa^2\norm{\vx-\vx'}^2\right) \\
\leq & 2\kappa^2\rho^2\norm{\vx-\vx'}^2,
\end{align*}
where the second inequality use Lemma~\ref{lem:P-smooth} that $\vy^*(\cdot)$ is $\kappa$-Lipschitz continuous.
In other words, the partial Hessian $\nabla^2_{xx}f(\cdot,\vy^*(\cdot))$ is $\sqrt{2}\kappa\rho$-Lipschitz continuous. Similarly, the partial Hessians $\nabla^2_{xy}f(\cdot,\vy^*(\cdot))$, $\nabla^2_{yx}f(\cdot,\vy^*(\cdot))$ and $\nabla^2_{yy}f(\cdot,\vy^*(\cdot))$ are also $\sqrt{2}\kappa\rho$-Lipschitz continuous.

Then we verify the Lipschitz continuity of  $\left(\nabla^2_{yy}f(\cdot,\vy^*(\cdot))\right)^{-1}$. 
For any $x$ and $x'$, we have
\begin{align*}
& \norm{\left(\nabla^2_{yy}f(\vx,\vy^*(\vx))\right)^{-1}-\left(\nabla^2_{yy}f(\vx',\vy^*(\vx'))\right)^{-1}} \\
=& \norm{\left(\nabla^2_{yy}f(\vx,\vy^*(\vx))\right)^{-1}\left(\nabla^2_{yy}f(\vx,\vy^*(\vx))-\nabla^2_{yy}f(\vx',\vy^*(\vx'))\right)\left(\nabla^2_{yy}f(\vx,\vy^*(\vx'))\right)^{-1}} \\
\leq & \frac{1}{\mu^2}\norm{\nabla^2_{yy}f(\vx,\vy^*(\vx))-\nabla^2_{yy}f(\vx',\vy^*(\vx')))} 
\leq  \frac{\sqrt{2}\kappa\rho}{\mu^2}\norm{\vx-\vx'}, 
\end{align*}
where we use the strongly-concavity of $f$ which implies $(\nabla_{yy}^2 f(\vx,\vy^*(\vx)))^{-1}\preceq \mu^{-1}\mI$, then $\left(\nabla^2_{yy}f(\cdot,\vy^*(\cdot))\right)^{-1}$ is $(\sqrt{2}\kappa\rho/\mu^2)$-Lipschitz continuous.

We use Lemma~\ref{lem:lip} with 
\begin{align*}
& \fH_1(\cdot) = \left(\nabla_{yy}^2f(\cdot,\vy^*(\cdot))\right)^{-1},~~~\fH_2(\cdot) = \nabla_{yx}^2 f(\cdot,\vy^*(\cdot)), \\
& \alpha_1 = \frac{\sqrt{2}\kappa\rho}{\mu^2},~~~~~\beta_1=\frac{1}{\mu},~~~~~\alpha_2 = \sqrt{2}\kappa\rho,~~~~~\beta_2=\ell,
\end{align*}
which means $\left(\nabla_{yy}^2 f(\cdot,\vy^*(\cdot))\right)^{-1}\nabla_{yx}^2 f(\cdot,\vy^*(\cdot))$ is $(2\sqrt{2}\kappa^2\rho/\mu)$-Lipschitz continuous, since 
\begin{align*}
\alpha_1\beta_2+\alpha_2\beta_1=\frac{\sqrt{2}\kappa^2\rho}{\mu}+\frac{\sqrt{2}\kappa\rho}{\mu} \leq \frac{2\sqrt{2}\kappa^2\rho}{\mu}.
\end{align*}
We also have $\norm{\left(\nabla_{yy}^2 f(\cdot,\vy^*(\cdot))\right)^{-1}\nabla_{yx}^2 f(\cdot,\vy^*(\cdot))}\leq \kappa$. 
We use Lemma~\ref{lem:lip} again with
\begin{align*}
& \fH_1(\cdot) = \nabla_{xy}^2 f(\cdot,\vy^*(\cdot)),~~~~\fH_2(\cdot) = \left(\nabla_{yy}^2f(\cdot,\vy^*(\cdot))\right)^{-1}\nabla_{yx}^2 f(\cdot,\vy^*(\cdot)) \\
& \alpha_1 = \sqrt{2}\kappa\rho,~~~~~~\beta_1=\ell,~~~~~~\alpha_2 = \frac{2\sqrt{2}\kappa^2\rho}{\mu},~~~~~~\beta_2=\kappa,
\end{align*}
then we obtain
\begin{align*}
\alpha_1\beta_2+\alpha_2\beta_1=\sqrt{2}\kappa^2\rho+2\sqrt{2}\kappa^3\rho\leq 3\sqrt{2}\kappa^3\rho    
\end{align*}
and
$\nabla_{xy}^2 f(\cdot,\vy^*(\cdot))\left(\nabla_{yy}^2 f(\cdot,\vy^*(\cdot))\right)^{-1}\nabla_{yx}^2 f(\cdot,\vy^*(\cdot))$ is $3\sqrt{2}\kappa^3\rho$-Lipschitz continuous.
Finally, we have
{\begin{align*}
& \norm{\nabla^2 P(\vx) - \nabla^2 P(\vx')} \\
\leq & \norm{\nabla^2_{xx} f(\vx,\vy^*(\vx)) - \nabla^2_{xx}f(\vx',\vy^*(\vx'))} \\
& + \big\|\nabla^2_{xy} f(\vx,\vy^*(\vx))(\nabla^2_{yy}  f(\vx,\vy^*(\vx)))^{-1}\nabla^2_{yx} f(\vx,\vy^*(\vx)) \\
&~~~~~~ -\nabla^2_{xy} f(\vx',\vy^*(\vx')) (\nabla^2_{yy} f(\vx',\vy^*(\vx')))^{-1}\nabla^2_{yx} f(\vx',\vy^*(\vx'))\big\| \\
\leq & \sqrt{2}\kappa\rho\norm{\vx-\vx'} + 3\sqrt{2}\kappa^3\rho\norm{\vx-\vx'} \\
= & 4\sqrt{2}\kappa^3\rho\norm{\vx-\vx'}.
\end{align*}}
\end{proof}

\section{Proofs for Section \ref{section:MCN}} \label{apsec:mcn}
In this section, we provide the proofs for theorems and lemmas in Section \ref{section:MCN}.

\subsection{Proof of Lemma \ref{lem:norm-s}}

\begin{proof}
Lemma \ref{lem:cubic-exact-solution} means
\begin{align*}
  &  \norm{\nabla P(\vx_{t+1})} \\
= &  \norm{\nabla P(\vx_{t+1})- \vg_t - \mH_t \vs_t - \frac{M}{2}\norm{\vs_t}\vs_t} \\
\leq &  \norm{\nabla P(\vx_{t+1}) - \nabla P(\vx_t) - \nabla^2 P(\vx_t)\vs_t} + \norm{\nabla P(\vx_t) -  \vg_t} + \norm{\nabla^2 P(\vx_t)\vs_t  - \mH_t \vs_t} + \frac{M}{2}\norm{\vs_t}^2 \\
\leq & \frac{M}{2}\norm{\vs_t}^2 + C_g\eps + C_H\sqrt{M\eps}\norm{\vs_t} + \frac{M}{2}\norm{\vs_t}^2 \\
= & M\norm{\vs_t}^2 + C_g\eps + C_H\sqrt{M\eps}\norm{\vs_t} \\
\leq & M\norm{\vs_t}^2 + C_g\eps + 
\frac{C_H\left(\eps + M\norm{\vs_t}^2\right)}{2} \\
= & \frac{(1+C_H)M}{2}\norm{\vs_t}^2 + \left(C_g + \frac{C_H}{2}\right)\eps, 
\end{align*}
which implies
\begin{align*}
\norm{\vs_t}^2 
\geq \frac{2}{(1+C_H)M}\left(\norm{\nabla P(\vx_{t+1})} - \left(C_g + \frac{C_H}{2}\right)\eps\right).
\end{align*}
We also have
\begin{align*}
    & \nabla^2 P(\vx_{t+1}) \\
\succeq & \mH_t - \norm{\mH_t - \nabla^2 P(\vx_{t+1})} \mI \\
\succeq & -\frac{M}{2}\norm{\vs_t}\mI - \norm{\mH_t - \nabla^2 P(\vx_{t+1})} \mI \\
\succeq & -\frac{M}{2}\norm{\vs_t}\mI - \norm{\mH_t - \nabla^2 P(\vx_{t})} \mI - \norm{\nabla^2 P(\vx_{t}) - \nabla^2 P(\vx_{t+1})} \mI \\
\succeq & -\frac{M}{2}\norm{\vs_t}\mI - C_H\sqrt{M\eps} \mI - M\norm{\vs_t} \mI \\
\succeq & -\frac{M+2M}{2}\norm{\vs_t}\mI - C_H\sqrt{M\eps} \mI,
\end{align*}
which implies
\begin{align*}
 \norm{\vs_t}  
\geq -\frac{2}{M+2M}\lambda_{\min}\left(\nabla^2 P(\vx_{t+1})\right) -  \frac{2C_H\sqrt{M}}{M+2M} \sqrt{\eps}.
\end{align*}
If $\vx_{t+1}$ is not an $(\eps,\sqrt{M\eps}\,)$-second-order stationary point, then
\begin{itemize}
\item If $\norm{\nabla P(\vx_{t+1})}\geq\eps$, we have
\begin{align*}
\norm{\vs_t} 
\geq \sqrt{\frac{2}{(1+C_H)M}\left(1-C_g - \frac{C_H}{2}\right)\eps} 
> \frac{1}{2}\sqrt{\frac{\eps}{M}}.
\end{align*}
\item If $-\lambda_{\min}\left(\nabla^2 P(\vx_{t+1})\right)\geq\sqrt{M\eps}$, we have
\begin{align*}
 \norm{\vs_t}  
\geq \frac{2}{3M}\sqrt{M\eps} -  \frac{2C_H\sqrt{M\eps}}{3M} 
> \frac{1}{2}\sqrt{\frac{\eps}{M}}.
\end{align*}
\end{itemize}
\end{proof}

\subsection{Proof of Theorem \ref{thm:MCN}}

\begin{proof}
Lemma~\ref{lem:norm-s} means we only needs to show that $T$ is no more than $\big\lceil 192(P(\vx_0) - P^*)\sqrt{M}\eps^{-1.5}\big\rceil$ by 
assuming $\norm{\vs_t^*}\geq \frac{1}{2}\sqrt{\eps/M}$ for all $t=1,\dots,T-1$. We have
\begin{align}\label{ieq:descent-P}
\begin{split}
& P(\vx_{t+1}) - P(\vx_t) \\
\leq &  \nabla P(\vx_t)^\top \vs_t^* + \frac{1}{2}(\vs_t^*)^\top \nabla^2 P(\vx_t) \vs_t^* + \frac{M}{6}\norm{\vs_t^*}^3 \\
= &  \vg_t^\top \vs_t^* + \frac{1}{2}(\vs_t^*)^\top \mH_t \vs_t^* + \frac{M}{6}\norm{\vs_t^*}^3  + (\nabla P(\vx_t)-\vg_t)^\top \vs_t^* + \frac{1}{2}(\vs_t^*)^\top (\nabla^2P(\vx_t)-\mH_t) \vs_t^* \\
\leq & -\frac{M}{12}\norm{\vs_t^*}^3  + C_g\eps\norm{\vs_t^*} + \frac{C_H\sqrt{M\eps}}{2}\norm{\vs_t^*}^2 \\
\leq & -\frac{M}{12}\norm{\vs_t^*}^3  + 4C_gM\norm{\vs_t^*}^3 + C_HM\norm{\vs_t^*}^3,
\end{split}
\end{align}
where the first inequality comes from (\ref{ieq:h-smooth2}) of Lemma~\ref{lem:h-smooth}; the second inequality comes from (\ref{ieq:cubic-sol-3}) of Lemma~\ref{lem:cubic-exact-solution} and Assumption~\ref{asm:error-g-h}; the last step is due to the assumption $\norm{\vs_t^*}\geq \frac{1}{2}\sqrt{\eps/M}$.
The result of (\ref{ieq:descent-P}) can be written as
\begin{align}\label{ieq:s-P}
\begin{split}
 & \frac{1}{8}\left(\frac{1}{12} - 4C_g - C_H\right)\sqrt{\frac{\eps^3}{M}} \\
\leq & \left(\frac{1}{12} - 4C_g - C_H\right)M\norm{\vs_t^*}^3 \\
\leq & P(\vx_t) - P(\vx_{t+1}).
\end{split}
\end{align}
Summing over inequality (\ref{ieq:s-P}) with $t=0,\dots, T-1$, we have
\begin{align*}
\frac{T}{192}\sqrt{\frac{\eps^3}{M}} \leq P(\vx_0) - P(\vx_t) \leq P(\vx_0) - P^*,
\end{align*}
which implies
\begin{align*}
T \leq 192(P(\vx_0) - P^*)\sqrt{M}\eps^{-3/2}.    
\end{align*}
Hence, Lemma~\ref{lem:norm-s} means the algorithm could find an $\big(\eps, \sqrt{M\eps}\,\big)$-second-order stationary point of $P(\vx)$ with at most $\big\lceil 192(P(\vx_0) - P^*)\sqrt{M}\eps^{-1.5}\big\rceil+1$ iterations.
\end{proof}

\subsection{Proof of Theorem \ref{thm:sum_K}}
\begin{proof}
We first use induction to show that 
\begin{align}\label{ieq:teps}
 \norm{\vy_t - \vy^*(\vx_t)}\leq\teps   
\end{align}
holds for any $t\geq 0$.
For $t=0$, Lemma~\ref{lem:agd} directly implies $\norm{\vy_0 - \vy^*(\vx_0)}\leq\teps$. Suppose it holds that $\norm{\vy_{t-1} - \vy^*(x_{t-1})}\leq\teps$ for any $t=t'-1$, then we have
\begin{align*}
 & \norm{\vy_{t'} - \vy^*(x_{t'})} \\
\leq & \sqrt{\kappa+1}\left(1 - \frac{1}{\sqrt{\kappa}}\right)^{K_{t'}/2}\norm{\vy_{t'-1} - \vy^*(\vx_{t'})}  \\
\leq & \sqrt{\kappa+1}\left(1 - \frac{1}{\sqrt{\kappa}}\right)^{K_{t'}/2}\big(\norm{\vy_{{t'}-1} - \vy^*(\vx_{{t'}-1})} + \norm{\vy^*(\vx_{{t'}-1}) - \vy^*(\vx_{t'})}\big)  \\
\leq & \sqrt{\kappa+1}\left(1 - \frac{1}{\sqrt{\kappa}}\right)^{K_{t'}/2}\left(\teps + \kappa\norm{\vx_{{t'}-1} - \vx_{t'}}\right)  \\
= & \sqrt{\kappa+1}\left(1 - \frac{1}{\sqrt{\kappa}}\right)^{K_{t'}/2}\left(\teps + \kappa\norm{\vs^*_{t'-1}}\right)  
\leq  \teps,
\end{align*}
where the first inequality is based on Lemma~\ref{lem:agd}; the second one use triangle inequality; the third one is based on induction hypothesis and the last step use the definition of $K_t$ and $\teps$.

Combining inequality~(\ref{ieq:teps}) with Lemma~\ref{lem:P-smooth}, Assumption~\ref{asm:g-smooth} and Assumption~\ref{asm:h-smooth}, we obtain 
\begin{align*}
  & \norm{\vg_t - \nabla P(\vx_t)} \\
= & \norm{\nabla_x f(\vx_t, \vy_t) - \nabla_x f(\vx_t, \vy^*(\vx_t))} \\
\leq & \ell\norm{\vy_t - \vy^*(\vx_t)} \leq   C_g\eps
\end{align*}
and
\begin{align*}
 & \norm{\nabla^2 P(\vx_t)-\mH_t}  \\
= & \norm{\nabla^2 f(\vx_t,\vy_t^*(\vx_t))-\nabla^2 f(\vx_t,\vy_t)} \\
\leq & \rho\norm{\vy_t-\vy^*(\vx_t)} \\
\leq & C_H\sqrt{M\eps}.
\end{align*}
The total gradient calls from AGD in Algorithm~\ref{alg:MCN} satisties 
\begin{align*}
\small\begin{split}
 &\sum_{t=0}^{T} K_t\\
\leq & 2\sqrt{\kappa}\left[K_0 + \sum_{t=1}^{T} \log\left(\sqrt{\kappa+1} + \frac{\kappa\sqrt{\kappa+1}}{\teps}\norm{\vs^*_{t-1}}\right)\right] +T+1 \\
= & \frac{2\sqrt{\kappa}}{3}\left[3K_0 + \sum_{t=1}^{T} \log\left(\sqrt{\kappa+1} + \frac{\kappa\sqrt{\kappa+1}}{\teps}\norm{\vs^*_{t-1}}\right)^3 \right] +T+1 \\
\leq & \frac{2\sqrt{\kappa}}{3}\left[3K_0 + \sum_{t=1}^{T}\log\left(8(\kappa+1)^{1.5} + \frac{8\kappa^3(\kappa+1)^{1.5}}{\teps^3}\norm{\vs^*_{t-1}}^3\right) \right] +T+1  \\
= & \frac{2\sqrt{\kappa}}{3}\left[3K_0+\log\left(\prod_{t=1}^{T}\left(8(\kappa+1)^{1.5} + \frac{8\kappa^3(\kappa+1)^{1.5}}{\teps^3}\norm{\vs^*_{t-1}}^3\right)\right)\right] +T+1   \\
\leq & \frac{2\sqrt{\kappa}}{3}\left[3K_0 + \log\left(\frac{1}{T}\sum_{t=1}^{T}\left(8(\kappa+1)^{1.5} + \frac{8\kappa^3(\kappa+1)^{1.5}}{\teps^3}\norm{\vs^*_{t-1}}^3\right)\right)^T\right] +T+1   \\
= & \frac{2\sqrt{\kappa}T}{3}\left[\frac{3}{T}\log\left(\frac{\sqrt{\kappa+1}}{\teps}\norm{\vy^*(\vx_0)}\right) + \log\left(8(\kappa+1)^{1.5} + \frac{8\kappa^3(\kappa+1)^{1.5}}{T\teps^3} \sum_{t=1}^{T}\norm{\vs^*_{t-1}}^3\right)\right] +T+1,
\end{split}
\end{align*}
where the first inequality is based on the fact $(a+b)^3\leq 8(a^3+b^3)$ for $a,b\geq 0$; the second inequality is based on AM–GM inequality.
\end{proof}

\subsection{Proof of Corollary~\ref{cor:complexity}}
\begin{proof}
The output is a desired second-order-stationary point can be proved by directly combining Theorem~\ref{thm:MCN} and Theorem~\ref{thm:sum_K}. Here we introduce $\eps'$ to eliminate the constant term $4\sqrt{2}$ in $M$.
Connecting the upper bound of $\sum_{t=1}^T K_t$ in Theorem~\ref{thm:sum_K} and inequality (\ref{ieq:s-P}) in the proof of Theorem~\ref{thm:MCN}, we have
{\small\begin{align*}
   \sum_{t=1}^{T} K_t  
\leq & \frac{2\sqrt{\kappa}T}{3}\log\left(\frac{3}{T}\log\left(\frac{\sqrt{\kappa+1}}{\teps}\norm{\vy^*(\vx_0)}\right)+8(\kappa+1)^{1.5} + \frac{192\kappa^3(\kappa+1)^{1.5}}{TM\teps'^3} \left(P(\vx_0)-P^*\right)\right) \\
= & \tilde\fO\left(\sqrt{\kappa M}\eps^{-1.5}\right) =   \tilde\fO\left(\kappa^2\sqrt{\rho}\eps^{-1.5}\right).
\end{align*}}\\
The claim follows from the fact that we call gradient oracle for $\fO(\sum_{t=1}^{T} K_t)$ times and perform Hessian (inverse) and exact cubic sub-problem solver calls for $\fO(T)$ times.
\end{proof}

\subsection{Proof of Corollary~\ref{cor:local-minimax}}

\begin{proof}
Following the the proof of Theorem~\ref{thm:sum_K}, Lemma \ref{lem:agd} means
\begin{align}\label{ieq:final-agd}
\begin{split}
 & \norm{\hy - \vy^*(\hx)} \\
\leq & \sqrt{\kappa+1}\left(1 - \frac{1}{\sqrt{\kappa}}\right)^{\hK/2}\norm{\vy_t - \vy^*(\hx)}  \\
\leq & \sqrt{\kappa+1}\left(1 - \frac{1}{\sqrt{\kappa}}\right)^{\hK/2}\big(\norm{\vy_t - \vy^*(\vx_t)} + \norm{\vy^*(\vx_t) - \vy^*(\hx)}\big)  \\
\leq & \sqrt{\kappa+1}\left(1 - \frac{1}{\sqrt{\kappa}}\right)^{\hK/2}\left(\teps + \kappa\norm{\vs_t^*}\right)  \\
\leq & \sqrt{\kappa+1}\left(1 - \frac{1}{\sqrt{\kappa}}\right)^{\hK/2}\left(\teps + \frac{\kappa}{{2^{2.25}}}\sqrt{\frac{\eps}{M}}\right)  \\
\leq & \min\left\{\frac{\alpha}{2\ell},~ \frac{\beta}{8\kappa^2\rho}\right\},
\end{split}
\end{align}
where we use $\norm{\vs_t^*}\leq \frac{1}{2}\sqrt{\eps/M}$ which is based on Lemma~\ref{lem:norm-s}.
Corollary~\ref{cor:complexity} means $\hx$ is an $\left(\eps, \kappa^{1.5}\sqrt{\rho\eps}\right)$-second-order stationary point of $P(\vx)$. Then using smoothness of $f$ and inequality (\ref{ieq:final-agd}), we have
\begin{align*}
 \norm{\nabla_x f(\hx, \hy)} 
\leq & \norm{\nabla_x f(\hx, \vy^*(\hx))} + \norm{\nabla_x f(\hx, \vy^*(\hx)) - \nabla_x f(\hx, \hy)} \\
\leq & \norm{\nabla P(\hx)} + \ell\norm{\vy^*(\hx) - \hy}  
= \frac{5\alpha}{6}
\end{align*}
and
\begin{align*}
 \norm{\nabla_y f(\hx,\hy)} 
= & \norm{\nabla_y f(\hx,\hy)-\nabla_y f(\hx,\vy^*(\hx))} 
\leq  \ell\norm{\hy- \vy^*(\hx)} \leq  \frac{\alpha}{2},
\end{align*}
which means $\norm{\nabla f(\hx,\hy)} \leq \alpha$. 
We also have
\begin{align*}
\small\begin{split}
& \nabla_{xx}^2 f(\hx, \hy) - \nabla_{xy}^2 f(\hx, \hy) \left(\nabla_{yy}^2 f(\hx, \hy)\right)^{-1} \nabla_{yx}^2f(\hx, \hy) \\
\succeq & \nabla_{xx}^2 P(\hx) - \norm{\nabla_{xx}^2 f(\hx, \hy) - \nabla_{xy}^2 f(\hx, \hy) \left(\nabla_{yy}^2 f(\hx, \hy)\right)^{-1} \nabla_{yx}^2f(\hx, \hy) - \nabla^2 P(\hx)}\mI \\
\succeq & \nabla_{xx}^2 P(\hx) - \norm{\nabla_{xx}^2 f(\hx, \hy) - \nabla_{xx}^2 f(\hx,\vy^*(\hx))}\mI \\
& {-} \norm{\nabla_{xy}^2 f(\hx, \hy) \left(\nabla_{yy}^2 f(\hx, \hy)\right)^{-1} \nabla_{yx}^2f(\hx, \hy) {-} \nabla_{xy}^2 f(\hx, \vy^*(\hx)) \left(\nabla_{yy}^2 f(\hx, \vy^*(\hx))\right)^{-1} \nabla_{yx}^2f(\hx, \vy^*(\hx))}\mI \\
\succeq & -\frac{\beta}{2}\mI - \kappa\rho\norm{\hy-\vy^*(\hx)}\mI  - 3\kappa^2\rho\norm{\hy-\vy^*(\hx)}\mI \\
\succeq & -\beta \mI,
\end{split}
\end{align*}
where the third inequality depends Lemma~\ref{lem:lip-another} that $\nabla^2_{xy}f(\vx,\cdot)\left(\nabla_{yy}^2 f(\vx, \cdot)\right)^{-1} \nabla_{yx}^2$ is $3\kappa^2\rho$-Lipschitz continuous.%, which is presented in appendix.
Hence, Definition~\ref{dfn:strict} means there exists a local minimax point $(\vx^*, \vy^*)$ such that $\norm{\hx-\vx^*}^2+\norm{\hy-\vy^*}^2\leq\gamma^2$.
\end{proof}

\section{Proof of Section~\ref{sec:imcn}} \label{apsec:imcn}
In this section, we provide the proofs for Theorems and Lemmas in Section~\ref{sec:imcn}.

\subsection{Proof of Lemma \ref{lem:Chebyshev}}
\begin{proof}
Since $\mX$ is symmetric positive definite and $\vzero \prec \mu' \mI\preceq \mX \preceq \ell' \mI \prec \mI$, we have~\citet[Section 8.6.1]{axelsson_1994}
\begin{align}\label{eq:invX-cheby}
    \mX^{-1} = \frac{c_0}{2}\mI+\sum_{k=1}^{\infty}c_k\mT_k(\mZ'). 
\end{align}
Chebyshev polynomials on scalar domain $-1 \leq \lambda \leq 1$ can be written as $T_k(\lambda)=\cos\left(k\cos^{-1}(\lambda)\right)$. Let eigenvalue decomposition of $\mZ'$ be $\mU{\bf\Lambda}\mU^\top$, then we have $\norm{\mT_k(\mZ')}=\norm{\mU\mT_k({\bf\Lambda})\mU^\top}\leq 1$. Combining definition of $c_k$ and (\ref{eq:invX-cheby}), we have
\begin{align*}
& \norm{\mX^{-1} - \left(\frac{c_0}{2}\mI+\sum_{k=1}^{K'}c_k\mT_k(\mZ')\right)}   
=  \norm{\sum_{k=K'+1}^{\infty}c_k\mT_k(\mZ')}  \\
\leq &  \frac{2}{\sqrt{\ell'\mu'}}\sum_{k=K'+1}^{\infty}\left|\frac{\sqrt{\mu'/\ell'}-1}{\sqrt{\mu'/\ell'}+1}\right|^k 
%= & \frac{2}{\sqrt{\ell'\mu'}}\left(\frac{\sqrt{\ell'/\mu'}-1}{\sqrt{\ell'/\mu'}+1}\right)^{K'+1}\sum_{k=0}^{\infty}\left(\frac{\sqrt{\ell'/\mu'}-1}{\sqrt{\ell'/\mu'}+1}\right)^k \\
= \frac{\sqrt{\ell'/\mu'}-1}{\sqrt{\ell'\mu'}}\left(1-\frac{2}{\sqrt{\ell'/\mu'}+1}\right)^{K'}. 
\end{align*}
\end{proof}

\subsection{Proof of Lemma~\ref{lem:inverse-approx}}
\begin{proof}
Recall that $-\ell \mI \preceq \nabla_{yy}^2 f(\vx_t,\vy_t)\preceq -\mu \mI$.
We estimate the inverse of Hessian of $\vy$ as 
\begin{align*}
\left(-\frac{1}{2\ell}\nabla_{yy}^2 f(\vx_t,\vy_t)\right)^{-1} \approx \frac{c_0}{2}\mI+\sum_{k=1}^{K'}c_k\mT_k(\mZ_t).
\end{align*}
Lemma~\ref{lem:Chebyshev} means
\begin{align*}
  \norm{\left(-\frac{1}{2\ell}\nabla_{yy}^2 f(\vx_t,\vy_t)\right)^{-1} - \left(\frac{c_0}{2}I+\sum_{k=1}^{K'}c_k\mT_k(\mZ_t)\right)} \leq 2(\kappa-\sqrt{\kappa})\left(1-\frac{2}{\sqrt{\kappa}+1}\right)^{K'}.
\end{align*}
Hence, we have
\begin{align}\label{ieq:inv-approx}
\small\begin{split}
& \norm{\nabla_{xy}^2 f(\vx_t,\vy_t)\left(-\nabla_{yy}^2 f(\vx_t,\vy_t)\right)^{-1}\nabla_{yx}^2 f(\vx_t,\vy_t) 
- \nabla_{xy}^2 f(\vx_t,\vy_t)\bigg(\frac{c_0}{4\ell}I+\frac{1}{2\ell}\sum_{k=1}^{K'}c_k\mT_k(\mZ_t)\bigg)\nabla_{yx}^2 f(\vx_t,\vy_t)} \\
\leq & \norm{\nabla_{xy}^2 f(\vx_t,\vy_t)}\norm{\left(-\nabla_{yy}^2 f(\vx_t,\vy_t)\right)^{-1}
- \left(\frac{c_0}{4\ell}I+\frac{1}{2\ell}\sum_{k=1}^{K'}c_k\mT_k(\mZ_t)\right)}\norm{\nabla_{yx}^2 f(\vx_t,\vy_t)} \\
\leq & \ell(\kappa-\sqrt{\kappa})\left(1-\frac{2}{\sqrt{\kappa}+1}\right)^{K'} \leq \kappa\ell\left(1-\frac{2}{\sqrt{\kappa}+1}\right)^{K'}
\end{split}
\end{align}
Consider that $\nabla_{xx}^2P(\vx)=\nabla_{xx}^2 f(\vx,\vy^*(\vx)) - \nabla_{xy}^2 f(\vx,\vy^*(\vx))\left(\nabla_{yy}^2 f(\vx,\vy^*(\vx))\right)^{-1}\nabla_{yx}^2 f(\vx,\vy^*(\vx))$ and we obtain $\vy_t\approx \vy^*(\vx_t)$ by AGD. Then we can bound the approximation error of $\mH_t$ as follows
\begin{align*}
& \norm{\nabla_{xx}^2P(\vx_t)-\left(\nabla_{xx}^2 f(\vx_t,\vy_t) + \nabla_{xy}^2 f(\vx_t,\vy_t)\left(\frac{c_0}{4\ell}I+\frac{1}{2\ell}\sum_{k=1}^{K'}c_k\mT_k(\mZ_t)\right)\nabla_{yx}^2 f(\vx_t,\vy_t)\right)} \\
\leq & \norm{\nabla_{xx}^2P(\vx_t) - \left(\nabla_{xx}^2 f(\vx_t,\vy_t) - \nabla_{xy}^2 f(\vx_t,\vy_t)\left(\nabla_{yy}^2 f(\vx_t,\vy_t)\right)^{-1}\nabla_{yx}^2 f(\vx_t,\vy_t)\right)} \\
& + \Bigg\|\nabla_{xx}^2 f(\vx_t,\vy_t) - \nabla_{xy}^2 f(\vx_t,\vy_t)\left(\nabla_{yy}^2 f(\vx_t,\vy_t)\right)^{-1}\nabla_{yx}^2 f(\vx_t,\vy_t) \\
&~~~~- \left(\nabla_{xx}^2 f(\vx_t,\vy_t) + \nabla_{xy}^2 f(\vx_t,\vy_t)\left(\frac{c_0}{4\ell}\mI+\frac{1}{2\ell}\sum_{k=1}^{K'}c_k\mT_k(\mZ_t)\right)\nabla_{yx}^2 f(\vx_t,\vy_t)\right)\Bigg\|_2 \\
\leq & 3\kappa^2\rho\norm{\vy_t-\vy^*(\vx_t)} + \kappa\ell\left(1-\frac{2}{\sqrt{\kappa}+1}\right)^{K'} \\
\leq & 3\rho\kappa^2\sqrt{\kappa+1}\left(1-\frac{1}{\sqrt{\kappa}}\right)^{K_t/2}\norm{\vy_{t-1}-\vy^*(\vx_t)} + \kappa\ell\left(1-\frac{2}{\sqrt{\kappa}+1}\right)^{K'},
\end{align*}
where the second inequality is according to $\nabla_{xy}^2 f(\vx, \cdot) \left(\nabla_{yy}^2 f(\vx, \cdot)\right)^{-1} \nabla_{yx}^2 f(\vx, \cdot)$ is $3\kappa^2\rho$-Lipschitz continuous (Lemma~\ref{lem:lip-another}) and the result of (\ref{ieq:inv-approx}); the last step is based on Lemma~\ref{lem:agd}.
\end{proof}

\subsection{Proof of Theorem \ref{thm:IMCN-T}}
We first introduce some lemmas.

\begin{lem}[{\citet[Lemma 7 and Lemma 11]{tripuraneni2017stochastic}}]\label{lem:sub-GD}
Suppose problem (\ref{prob:cubic-sub}) holds that $\norm{\vg}\leq L^2/M$ and $\norm{\mH} \leq (1+\eps_H)L$ for some $\eps_H \leq 3/25$. Running Algorithm~\ref{alg:cubic-GD} with
\begin{align*}
\eta = \frac{1}{20L},~~
\sigma=\frac{C_\sigma M^2\sqrt{\eps^3/M^3}}{4 608(4L+\sqrt{M\eps}\,)}
\end{align*}
and
\begin{align*}
\fK(\eps,\delta')=& \frac{19200L}{C_\sigma \sqrt{M\eps}}\left[6\log\left(\!3{+}\frac{9\sqrt{d_x}}{\delta'}\!\right) {+} 18\log\left(\frac{6L}{\sqrt{M\eps}}\right){+} 14\log\bigg(\!\frac{48(L + C_H\sqrt{M\eps}\,)}{C_\sigma \sqrt{M\eps}}+\frac{24}{C_\sigma}\!\bigg)\right]\\
=& \tilde\fO\left(\frac{L}{\sqrt{M\eps}}\right),
\end{align*}
where $C_\sigma>0$.
Then with probability $1-\delta'$, the output satisfies 
\begin{align*}
\norm{\hs} \leq \norm{\vs^*}+\frac{\sqrt{C_\sigma}}{24}\sqrt{\frac{\eps}{M}}.
\end{align*}
If we further suppose $\norm{\vs^*}\geq \frac{1}{2}\sqrt{\eps/M}$, it holds that
\begin{align*} 
\tm(\hs)\leq \tm(\vs^*) + \frac{C_\sigma M}{12}\norm{\vs^*}^3 
\text{~~and~~} \Delta=\tm(\hs) \leq -\frac{1-C_\sigma}{96}\sqrt{\frac{\eps^3}{M}}.
\end{align*}
\end{lem}

\begin{proof}

Most of the results in this lemma can be obtain by Lemma 7 and Lemma 11 of \citet{tripuraneni2017stochastic} directly. Here, we should the derivation of detailed expression for $\fK(\eps,\delta')$. We follow the proof of~\citet{tripuraneni2017stochastic}'s Lemma 11 to show it. Using notations in this paper, we have $\frac{L'^3}{3\norm{\vs_t^*}^3}\leq\bar\sigma\leq 1$, $\eta=1/20L$, $L'=\frac{1}{2}\sqrt{\eps/M}$, $\eps'=\frac{C_\sigma M}{24}\norm{\vs_t^*}^3$, $\norm{\vs^*}\geq\sqrt{\eps/M}=L'$ and $\gamma=-\lambda_{\min}(\mH)$. The notation $\fK(\eps,\delta')$ corresponds to $\fT(\eps)$ of~\citet{tripuraneni2017stochastic}.
The optimal condition of cubic sub-problem means $\norm{\vs_t^*}\leq 3L/M$.
We have
\begin{align*}
\min\left\{\frac{1}{M\norm{\vs_t^*}-\gamma}, \frac{10\norm{\vs_t^*}^2}{\eps'}\right\} =  \min\left\{\frac{1}{M\norm{\vs_t^*}-\gamma}, \frac{240}{C_\sigma M \norm{\vs_t^*}}\right\},
\end{align*}
then
{\small\begin{align*}
  &  \fK(\eps,\delta') \\
= & \frac{1+\bar\sigma}{\eta}\min\left\{\frac{1}{M\norm{\vs_t^*}-\gamma}, \frac{10\norm{\vs_t^*}^2}{\eps'}\right\}\left[6\log\bigg(1+\BI_{\{\gamma>0\}}\frac{3\sqrt{d}}{\bar\sigma\delta'}\bigg) + 14\log\left(\frac{(\beta+M\norm{\vs_t^*})\norm{\vs_t^*}^2}{\eps'}\right)\right] \\
\leq & \frac{2}{1/20L}\cdot\frac{240}{C_\sigma M \norm{\vs_t^*}}\left[6\log\left(1+\frac{3\sqrt{d}}{\bar\sigma\delta'}\right) + 14\log\left(\frac{(\beta+M\norm{\vs_t^*})\norm{\vs_t^*}^2}{\eps'}\right)\right] \\
= & \frac{9600L}{C_\sigma M \norm{\vs_t^*}}\left[6\log\left(3\bar\sigma+\frac{9\sqrt{d}}{\delta'}\right) + 6\log\left(\frac{1}{3\bar\sigma}\right)+ 14\log\left(\frac{24\beta\norm{\vs_t^*}^2}{C_\sigma M \norm{\vs_t^*}^3}+\frac{24M\norm{\vs_t^*}^3}{C_\sigma M \norm{\vs_t^*}^3}\right)\right] \\
\leq & \frac{9600L}{C_\sigma M \frac{1}{2}\sqrt{\eps/M}}\left[6\log\left(3+\frac{9\sqrt{d}}{\delta'}\right) + 18\log\left(\frac{\norm{\vs_t^*}}{L'}\right)+ 14\log\left(\frac{24\beta}{C_\sigma M \frac{1}{2}\sqrt{\eps/M}}+\frac{24}{C_\sigma}\right)\right] \\
\leq & \frac{19200L}{C_\sigma \sqrt{M\eps}}\left[6\log\left(3+\frac{9\sqrt{d}}{\delta'}\right) + 18\log\left(\frac{6L}{\sqrt{M\eps}}\right)+ 14\log\left(\frac{48(L + C_H\sqrt{\rho\eps}\,)}{C_\sigma \sqrt{M\eps}}+\frac{24}{C_\sigma}\right)\right] \\
= & \tilde\fO\left(\frac{L}{\sqrt{M\eps}}\right).
\end{align*}}
\end{proof}

The following lemma corresponds to the case of calling Algorithm~\ref{alg:cubic-GD} when $\norm{\vg}\geq L^2/M$. It extends Lemma 9 of \citet{tripuraneni2017stochastic}, leading to the result includes term $\norm{\vs_t}$ additionally, which is useful to bound the number of gradient calls from AGD in further analysis. 
\begin{lem}\label{lem:sub-Cauchy}
Using the notation of Algorithm~\ref{alg:IMCN}, we suppose $\norm{\vg_t}\geq L^2/M$ and Assumption~\ref{asm:error-g-h} holds with $C_H\leq 1/40$ and $C_g\leq 7/600$ at $t$-th iteration. Let $\vs_t=-R_C \vg_t/\norm{\vg_t}$, where 
\begin{align*}
R_C=-\dfrac{\vg_t^\top \mH_t\vg_t}{M\norm{\vg_t}^2} + \sqrt{\left(\dfrac{\vg_t^\top \mH_t \vg_t}{M\norm{\vg_t}^2}\right)^2 + \dfrac{2\norm{\vg_t}}{M}}.
\end{align*}
Then we have 
\begin{align*}
\Delta_t=\tm_t(\vs_t) \leq -\frac{7}{20}\sqrt{\frac{\eps^3}{M}} ~~~\text{and}~~~
\frac{M}{24}\norm{\vs_t}^3 \leq  -\frac{7}{20}\sqrt{\frac{\eps^3}{M}}  + P(\vx_t) - P(\vx_t+\vs_t).
\end{align*}
\end{lem}
\begin{proof}
The assumption on $\norm{\vg_t}$ means
\begin{align*}
R_C =& -\dfrac{\vg_t^\top \mH_t\vg_t}{M\norm{\vg_t}^2} + \sqrt{\left(\dfrac{\vg_t^\top \mH_t \vg_t}{M\norm{\vg_t}^2}\right)^2 + \dfrac{2\norm{\vg_t}}{M}} \\
\geq & \frac{1}{M}\left(-\dfrac{\vg_t^\top \mH_t\vg_t}{\norm{\vg_t}^2} + \sqrt{\left(\dfrac{\vg_t^\top \mH_t \vg_t}{\norm{\vg_t}^2}\right)^2 + 2L^2}\right).
\end{align*}
Note that function $-x+\sqrt{x^2+2}$ is decreasing on $\BR$ and 
$\norm{\nabla^2 P(\vx_t)}\leq L$. Combining conditions of $\norm{\nabla^2 P(\vx_t)-\mH_t} \leq~C_H\sqrt{M\eps}$ and $\eps\leq L^2/M$, we obtain $\norm{(\nabla^2 P(\vx_t)-\mH_t)\vg_t} \leq   C_H L\norm{\vg_t}$, which implies $\vg_t^\top \mH_t \vg_t \leq (1+C_H)L\norm{\vg_t}^2$. Hence, for all $C_H\leq 1/40$, we have
\begin{align}\label{bound:lower-RC}
\begin{split}
R_C \geq & \frac{L}{M}\left(-(1+C_H) + \sqrt{(1+C_H)^2 + 2}\right) \\
\geq & \frac{L}{M}\left(-\left(1+\frac{1}{40}\right) + \sqrt{\left(1+\frac{1}{40}\right)^2 + 2}\right) \geq \frac{7L}{10M}.
\end{split}
\end{align}
Observe that we have $R_C=\argmin_{\eta\in{\BR^+}}\tm_t(-\eta \vg_t/\norm{\vg_t})$, which means the derivative of function $q_t(\eta)=\tm_t(-\eta\cdot \vg_t/\norm{\vg_t})$ is 0 at $\eta=R_C$. Hence, we have
\begin{align}\label{bound:tm-s-eps}
\small\begin{split}
\tm_t\left(s_t\right) 
= \tm_t\left(-\frac{R_C\vg_t}{\norm{\vg_t}}\right) 
= -\frac{R_C\norm{\vg_t}}{2} - \frac{MR_C^3}{12} 
\leq -\frac{7L^3}{20M^2}  - \frac{MR_C^3}{12}  
\leq - \frac{7}{20}\sqrt{\frac{\eps^3}{M}}- \frac{MR_C^3}{12}.
\end{split}
\end{align}
Using inequality (\ref{ieq:h-smooth2}) of Lemma~\ref{lem:h-smooth}, we have
\begin{align}\label{bound:Pt_Pt1}
\begin{split}
& P(\vx_t + \vs_t) - P(\vx_t) \\
\leq &  \nabla P(\vx_t)^\top \vs_t + \frac{1}{2}(\vs_t)^\top \nabla^2 P(\vx_t) \vs_t + \frac{M}{6}\norm{\vs_t}^3 \\
= &  \vg_t^\top \vs_t + \frac{1}{2}(\vs_t)^\top \mH_t \vs_t + \frac{M}{6}\norm{\vs_t}^3  + (\nabla P(\vx_t)-\vg_t)^\top \vs_t + \frac{1}{2}(\vs_t)^\top (\nabla^2P(\vx_t)-\mH_t) \vs_t \\
\leq & \tm_t\left(\vs_t\right) + C_g\eps\norm{\vs_t} + \frac{C_H\sqrt{M\eps}}{2}\norm{\vs_t}^2 \\
= & -\frac{R_C\norm{\vg_t}}{2} - \frac{MR_C^3}{12}  + C_g\eps\norm{\vs_t} + \frac{C_H\sqrt{M\eps}}{2}\norm{\vs_t}^2 \\
= & -\frac{R_C\norm{\vg_t}}{2} - \frac{MR_C^3}{24} - \left(\frac{MR_C^3}{24}  - C_g\eps\norm{\vs_t} - \frac{C_H\sqrt{M\eps}}{2}\norm{\vs_t}^2\right) \\
\leq & -\frac{7}{20}\sqrt{\frac{\eps^3}{M}} - \frac{MR_C^3}{24} - \frac{L^2R_C}{24M}\left(\left(\frac{MR_C}{L}\right)^2 - 24C_g - 12C_H\left(\frac{MR_C}{L}\right)\right).
\end{split}
\end{align}
Consider that function $q(x)=x^2-24C_g-12C_Hx$ is minimized at $7/10$ on $\{x:x\geq 7/10\}$ when $C_H\leq 1/40$. Since inequality (\ref{bound:lower-RC}) means $MR_C/L\geq 7/10$, we have
\begin{align*}
 -\left(\left(\frac{MR_C}{L}\right)^2 - 24C_g - 12C_H\left(\frac{MR_C}{L}\right)\right) 
\leq  -\left(\frac{7}{10}\right)^2 + 24C_g + 12C_H\left(\frac{7}{10}\right)  \leq 0
\end{align*}
when $C_H\leq1/40$ and $C_g\leq 7/600$.
Combining all above results, we have
\begin{align*}
 \frac{M}{24}\norm{\vs_t}^3 \leq  -\frac{7}{20}\sqrt{\frac{\eps^3}{M}}  + P(\vx_t) - P(\vx_t+\vs_t).
\end{align*}
\end{proof}
 
Since we approximate $\nabla P(\vx_t)$ and $\nabla^2 P(\vx_t)$ by $\vg_t$ and $\mH_t$ respectively, the procedure of Algorithm~\ref{alg:IMCN} can be viewed as solving nonconvex optimization problem~$\min_{\vx\in\BR^{d_x}}P(\vx)$ by inexact first-order and second-order information. Hence, the following lemma holds for our algorithm.

\begin{lem}[{\citet[Claim 2]{tripuraneni2017stochastic}}]\label{lem:sub-GD2}
Suppose that Assumption~\ref{asm:error-g-h} holds with   $C_g\leq1/240$ and $C_H\leq1/200$; and we have $\norm{\vg_t}\leq L^2/M$ and $\tm_t(\vs_t)\leq-\frac{1-C_\sigma}{96}\sqrt{\eps^3/M}$ with $C_\sigma\leq1/4$. We have
\begin{align*}
    P(\vx_t+\vs_t)-P(\vx_t) \leq & -\frac{M}{32}\norm{\vs_t^*}^3 -  \frac{1}{625}\sqrt{\frac{\eps^3}{M}}.
\end{align*}
\end{lem}
\begin{proof}
We directly use the proof of Claim 2 of ~\citet{tripuraneni2017stochastic}.
The notation of $c_1$, $c_2$, $c_3$ and $c_4$ of \citet{tripuraneni2017stochastic}'s paper corresponds to $C_g$, $C_H$, $C_\sigma$ and $\sqrt{C_\sigma}/24$ here. If $\norm{\vs_t^*}\geq\frac{1}{2}\sqrt{\eps/M}$, we have
{\small\begin{align*}
& P(\vx_{t+1}) - P(\vx_t) \\
\leq & - \left(\frac{1-C_\sigma-48(C_g+C_H\sqrt{C_\sigma}/24) - 12C_HC_\sigma/576}{12}\right)M\norm{\vs_t}^* + \left(C_g+\frac{C_H\sqrt{C_\sigma}}{48}\right)\frac{C_\sigma}{24}\sqrt{\frac{\eps^3}{M}} \\
\leq & - \left(\frac{1-1/4-48(1/240+1/9600) - 3/576/200}{12}\right)M\norm{\vs_t}^* + \left(\frac{1}{240}+\frac{1/2}{9600}\right)\frac{1}{96}\sqrt{\frac{\eps^3}{M}} \\
\leq & -\frac{9M}{20}\norm{\vs_t^*}^3  + \frac{1}{10000}\sqrt{\frac{\eps^3}{M}} 
\leq -\frac{2M}{5}\norm{\vs_t^*}^3 - \frac{1}{150}\sqrt{\frac{\eps^3}{M}}. \end{align*}}\\
If $\norm{\vs_t^*}\leq\frac{1}{2}\sqrt{\eps/M}$, we have
{\small\begin{align*}
& P(\vx_{t+1}) - P(\vx_t) \\
\leq & -\frac{1-C_\sigma}{96}\sqrt{\frac{\eps^3}{M}} {+} \frac{(C_g+C_H\sqrt{C_\sigma}/24)\eps}{2}\sqrt{\frac{\eps}{M}} {+} \frac{C_HC_\sigma/576\sqrt{M\eps^3}}{8M} {+} \Big(C_g{+}\frac{C_H\sqrt{C_\sigma}/24}{2}\Big)\frac{\sqrt{C_\sigma}}{24}\sqrt{\frac{\eps^3}{M}} \\
\leq & -\frac{1}{256}\sqrt{\frac{\eps^3}{M}} -  \frac{1}{256}\sqrt{\frac{\eps^3}{M}} + \frac{1/240+1/9600}{2}\sqrt{\frac{\eps^3}{M}} + \frac{1/800}{4608}\sqrt{\frac{\eps^3}{M}} + \Big(1/240+\frac{1/9600}{2}\Big)\frac{1}{48}\sqrt{\frac{\eps^3}{M}} \\
\leq &  -\frac{1}{256}\sqrt{\frac{\eps^3}{M}} -  \frac{1}{256}\sqrt{\frac{\eps^3}{M}} + \frac{23}{10000}\sqrt{\frac{\eps^3}{M}} 
\leq   -\frac{M}{32}\norm{\vs_t^*}^3 -  \frac{1}{625}\sqrt{\frac{\eps^3}{M}}. 
\end{align*}}\\
We finish the proof by combing above results.
\end{proof}

We first bound the norm of $\vs_t$.
\begin{lem}\label{thm:IMCN-s}
Under the setting of Theorem~\ref{thm:IMCN-T},
if it satisfies $\Delta_t \leq -\frac{1}{128}\sqrt{\eps^3/M}$,
we have
\begin{align*}
    \frac{M}{256}\norm{\vs_t}^3  \leq   P(\vx_t) - P(\vx_t+\vs_t)  - \frac{1}{626}\sqrt{\frac{\eps^3}{M}}.
\end{align*}
with probability $1-\delta'$.
\end{lem}
\begin{proof}
In the case of $\norm{\vg_t}\geq L^2/M$, Lemma~\ref{lem:sub-Cauchy} implies
\begin{align}\label{bound:s1}
  \frac{M}{24}\norm{\vs_t}^3 \leq  -\frac{7}{20}\sqrt{\frac{\eps^3}{M}}  + P(\vx_t) - P(\vx_{t+1}).
\end{align}
In the case of $\norm{\vg_t}\leq L^2/M$, using the condition $\Delta_t \leq -\frac{1}{128}\sqrt{\eps^3/M}$ and
Lemma~\ref{lem:sub-GD2}, we have
\begin{align}\label{bound:s2}
    P(\vx_t+\vs_t)-P(\vx_t) \leq & -\frac{M}{32}\norm{\vs_t^*}^3 - \frac{1}{625}\sqrt{\frac{\eps^3}{M}}.
\end{align}
On the other hand, Lemma~\ref{lem:sub-GD} implies with probability $1-\delta'$, we have
\begin{align*}
\norm{\vs_t} \leq \norm{\vs_t^*}+\frac{\sqrt{C_\sigma}}{24}\sqrt{\frac{\eps}{M}} 
~~~\Longrightarrow~~~
M\norm{\vs_t}^3 \leq 8M\norm{\vs_t^*}^3+\frac{8C_\sigma^{1.5}M}{24^3}\sqrt{\frac{\eps^3}{M^3}},
\end{align*}
which means
\begin{align}\label{bound:ss}
- \frac{M}{32}\norm{\vs_t^*}^3  \leq  -  \frac{M}{256}\norm{\vs_t}^3 +  \frac{1}{3538944}\sqrt{\frac{\eps^3}{M}}. 
\end{align}
Connecting inequalities (\ref{bound:s2}) and (\ref{bound:ss}), we have
\begin{align}\label{bound:s3}
    P(\vx_t+\vs_t)-P(\vx_t) \leq & -\frac{M}{256}\norm{\vs_t}^3 -  \frac{1}{626}\sqrt{\frac{\eps^3}{M}}.
\end{align}
Combining the results of~(\ref{bound:s1}) and (\ref{bound:s3}), we prove this lemma.
\end{proof}

Then we give the proof of Theorem \ref{thm:IMCN-T}.

\begin{proof}
To bound the number of iterations, we are only interested in the iteration with $\Delta_t\leq-\frac{1}{128}\sqrt{\eps^3/M}$, otherwise the condition in line 7 holds and the algorithm will break the loop.
In such case, Lemma~\ref{thm:IMCN-s} mean
\begin{align*}
    \frac{M}{256}\norm{\vs_t}^3  \leq   P(\vx_t) - P(\vx_t+\vs_t)  - \frac{1}{626}\sqrt{\frac{\eps^3}{M}}.
\end{align*}
Suppose the total number of iteration is $T'$. Summing over $t=0,\dots,T'$, we have
\begin{align*}
T' \leq   626(P(\vx_0) - P(\vx_T))\sqrt{M}\eps^{-1.5}  
\leq   626(P(\vx_0) - P^*)\sqrt{M}\eps^{-1.5} 
=   \fO\left(\kappa^{1.5}\sqrt{\rho}\eps^{-1.5}\right). 
\end{align*}

Lemma~\ref{lem:norm-s} says if
$\vx_t+\vs_t^*$ is not an $\big(\eps, \sqrt{M\eps}\,\big)$-second-order stationary point of $P(\vx)$, then we have $\norm{\vs_t^*}\geq\frac{1}{2}\sqrt{\eps/M}$. Combining Lemma~\ref{lem:sub-GD} and Lemma~\ref{lem:sub-Cauchy}, we have
\begin{align*} 
\Delta_t = \tm_t(\vs_t) \leq \max\left\{-\frac{7}{20}, -\frac{1-C_\sigma}{96}\right\}\sqrt{\frac{\eps^3}{M}}
\leq -\frac{1}{128}\sqrt{\frac{\eps^3}{M}}.
\end{align*}
Hence, if condition in line 7 holds, we conclude that $\vx_t+\vs_t^*$ is an $\big(\eps, \sqrt{M\eps}\,\big)$-second-order stationary point and Lemma 8 of~\citet{tripuraneni2017stochastic} means the output $\hx=\vx_t+\hs$ is a $(\eps,2\kappa^{1.5}\sqrt{\rho\eps}\,)$-second-order stationary point of $P(\vx)$.

Note that Lemma~\ref{lem:sub-GD} means each iteration of our algorithm succeed with probability $1-\delta'$. Let $\fF_t$ be the event that $t$-th iteration fail. The union bound implies 
\begin{align*}
    \BP \left(\bigcup_{t=1}^T \fF_t\right) \leq \sum_{t=1}^T\BP \left( \fF_t\right) = T\delta' = \delta.
\end{align*}
Hence, the probability of success is at least $1-\delta$.
\end{proof}

\subsection{Proof of Theorem \ref{thm:IMCN-sum_K}}

\begin{proof}
We first use induction to show that 
\begin{align}\label{ieq:teps1}
 \norm{\vy_t - \vy^*(\vx_t)}\leq\teps   
\end{align}
holds for any $t\geq 0$.
For $t=0$, Lemma~\ref{lem:agd} directly implies $\norm{\vy_0 - \vy^*(\vx_0)}\leq\teps$. Suppose it holds that $\norm{\vy_{t-1} - \vy^*(x_{t-1})}\leq\teps$ for any $t=t'-1$, then we have
\begin{align}\label{ieq:teps2}
\begin{split}
 & \norm{\vy_{t'} - \vy^*(x_{t'})} \\
\leq & \sqrt{\kappa+1}\left(1 - \frac{1}{\sqrt{\kappa}}\right)^{K_{t'}/2}\norm{\vy_{t'-1} - \vy^*(\vx_{t'})}  \\
\leq & \sqrt{\kappa+1}\left(1 - \frac{1}{\sqrt{\kappa}}\right)^{K_{t'}/2}\big(\norm{\vy_{{t'}-1} - \vy^*(\vx_{{t'}-1})} + \norm{\vy^*(\vx_{{t'}-1}) - \vy^*(\vx_{t'})}\big)  \\
\leq & \sqrt{\kappa+1}\left(1 - \frac{1}{\sqrt{\kappa}}\right)^{K_{t'}/2}\left(\teps + \kappa\norm{\vx_{{t'}-1} - \vx_{t'}}\right)  \\
= & \sqrt{\kappa+1}\left(1 - \frac{1}{\sqrt{\kappa}}\right)^{K_{t'}/2}\left(\teps + \kappa\norm{\vs_{t'-1}}\right)  
\leq  \teps,
\end{split}
\end{align}
where the first inequality is based on Lemma~\ref{lem:agd}; the second one use triangle inequality; the third one is based on induction hypothesis and the last step use the definition of $K_t$ and $\teps$.

Combining inequality~(\ref{ieq:teps1}) with Lemma~\ref{lem:P-smooth} and Assumption~\ref{asm:g-smooth}, we obtain 
\begin{align*}
  \norm{\vg_t - \nabla P(\vx_t)} 
= \norm{\nabla_x f(\vx_t, \vy_t) - \nabla_x f(\vx_t, \vy^*(\vx_t))} 
\leq \ell\norm{\vy_t - \vy^*(\vx_t)} \leq   C_g\eps.
\end{align*}
Combining inequality~(\ref{ieq:teps1}) with Lemma \ref{lem:inverse-approx} and Assumption~\ref{asm:h-smooth},  we have
\begin{align*}
  & \norm{\nabla_{xx}^2P(\vx_t)-\mH_t}   \\
\leq & 3\rho\kappa^2\sqrt{\kappa+1}\left(1-\frac{1}{\sqrt{\kappa}}\right)^{K_t/2}\norm{\vy_{t-1}-\vy^*(\vx_t)} + \kappa\ell\left(1-\frac{2}{\sqrt{\kappa}+1}\right)^{K'} \\
\leq & 3\rho\kappa^2\teps + \frac{C_H\sqrt{M\eps}}{2} 
 \leq C_H\sqrt{M\eps}.
\end{align*}
We obtain the upper bound of $\sum_{t=0}^{T} K_t$ by following the analysis of Theorem~\ref{thm:sum_K}.
\end{proof}

\subsection{Proof of Corollary \ref{cor:complexity2}}

\begin{proof}
The output is a desired second-order-stationary point can be proved by directly combining Theorem~\ref{thm:IMCN-T} and~Theorem \ref{thm:IMCN-sum_K}. 
Connecting the upper bounds of $\sum_{t=1}^T K_t$ and Theorem~\ref{thm:IMCN-sum_K} , we conclude the total number of gradient calls is at most
\begin{align*}
   &\sum_{t=1}^{T} K_t  \\
\leq & \frac{2\sqrt{\kappa}T}{3}\log\left(\frac{3}{T}\log\left(\frac{\sqrt{\kappa+1}}{\teps}\norm{\vy^*(\vx_0)}\right) + 8(\kappa+1)^{1.5} + \frac{2048\kappa^3(\kappa+1)^{1.5}}{626TM\teps'^3} \left(P(\vx_0)-P^*\right)\right)  \\
= & \tilde\fO\left(\kappa^2\sqrt{\rho}\eps^{-1.5}\right).
\end{align*}

The total number of Hessian-vector calls from Algorithm~\ref{alg:cubic-GD} is at most
\begin{align*}
& T\cdot\fK(\eps,\delta')\cdot K' \\
\leq & \tilde\fO\left(\kappa^{1.5}\sqrt{\rho}\eps^{-1.5}\right)\cdot\tilde\fO\left(\frac{L}{\sqrt{M\eps}}\right)\cdot\tilde\fO\left(\sqrt{\kappa}\right) \\
= & \tilde\fO\left(\kappa^{1.5}\sqrt{\rho}\eps^{-1.5}\right)\cdot\tilde\fO\left(\frac{\kappa\ell}{\kappa^{1.5}\sqrt{\rho\eps}}\right)\cdot\tilde\fO\left(\sqrt{\kappa}\right) 
= \tilde\fO\left(\kappa^{1.5}\ell\eps^{-2}\right).
\end{align*}
Using Lemma 8 of~\citet{tripuraneni2017stochastic}, we know
the total number of Hessian-vector calls from Algorithm~\ref{alg:cubic-final} is at most 
\begin{align*}
\tilde\fO\left(\sqrt{\kappa}\right)\cdot\fO\left(\frac{L^2}{M\eps}\right)
= \tilde\fO\left(\frac{\ell^2}{\sqrt{\kappa}\rho\eps}\right),
\end{align*}
which is not the leading term in total complexity for small $\eps$.
\end{proof}

\section{Implementation of IMCN} \label{apsec:imple}
In the implementation of IMCN (the cubic sub-problem solver), all of steps related to $\mH_t$ can be view as computing Hessian-vector product of the form 
\begin{align}\label{eq:Hu}
 \mH_t\vu' = \left(\nabla_{xx}^2 f(\vx_t,\vy_t) + \nabla_{xy}^2 f(\vx_t,\vy_t)\left(\frac{c_0}{4\ell}\mI+\frac{1}{2\ell}\sum_{k=1}^{K'}c_k\mT_k(\mZ_t)\right)\nabla_{yx}^2 f(\vx_t,\vy_t)\right) \vu',
\end{align}
where $\vu'\in\BR^{d_x}$ is given. Recall that Chebyshev polynomial satisfies
\begin{align*}
\mT_k(\mZ_t)=2\mZ_t\mT_{k-1}(\mZ_t) - \mT_{k-2}(\mZ_t),    
\end{align*} 
which allows us computing (\ref{eq:Hu}) without constructing any Hessian matrix explicitly. Concretely, For fixed $t$, we define
\begin{align*}
\vu=\nabla_{yx}^2 f(\vx_t,\vy_t) \vu', \quad \mB_k=\frac{c_0}{4\ell}\mI+\frac{1}{2\ell}\sum_{k=1}^{K}c_k\mT_k(\mZ_t), \quad \vv_k=\mT_k(\mZ_t)\vu \quad \text{and} \quad \vw_k=\mB_k\vu. 
\end{align*}
Since we have $\mZ_t=\frac{4\ell}{\ell-\mu}\left(-\frac{1}{2\ell}\nabla_{yy}^2 f(\vx_t,\vy_t)-\frac{\ell+\mu}{4\ell}\mI\right)$, it holds that
\begin{align}\label{update:wk}
\begin{split}
  \vw_k = & \left(\mB_{k-1} + c_{k}\mT_k(\mZ_t)\right)\vu \\
= & \vw_{k-1} + c_{k}\left(2\mZ_t\mT_{k-1}(\mZ_t) - \mT_{k-2}(\mZ_t)\right)\vu \\
= & \vw_{k-1} + c_{k}\left(\frac{8\ell}{\ell-\mu}\left(-\frac{1}{2\ell}\nabla_{yy}^2 f(\vx_t,\vy_t)-\frac{\ell+\mu}{4\ell}\mI\right)\vv_{k-1} - \vv_{k-2}\right) \\
= & \vw_{k-1} - \frac{4c_{k}}{\ell-\mu}\nabla_{yy}^2 f(\vx_t,\vy_t)\vv_{k-1} - \frac{2c_{k}(\ell+\mu)}{\ell-\mu}\vv_{k-1}  - c_{k}\vv_{k-2}  \\
\end{split}
\end{align}
and
\begin{align}\label{update:vk}
\begin{split}
 \vv_{k} = & \mT_{k}(\mZ_t)\vu  \\
=& 2\mZ_t\mT_{k-1}(\mZ_t)\vu - \mT_{k-2}(\mZ_t)\vu \\
=& \frac{8\ell}{\ell-\mu}\left(-\frac{1}{2\ell}\nabla_{yy}^2 f(\vx_t,\vy_t)-\frac{\ell+\mu}{4\ell}\mI\right)\vv_{k-1}  - \vv_{k-2} \\
=&  -\frac{4}{\ell-\mu}\nabla_{yy}^2 f(\vx_t,\vy_t)\vv_{k-1} - \frac{2(\ell+\mu)}{\ell-\mu}\vv_{k-1}  - \vv_{k-2}. 
\end{split}
\end{align}
Hence, based on update rules (\ref{update:wk}) and (\ref{update:vk}), we can obtain $\mH_t\vu'$ with $\fO(K')=\tilde\fO(\sqrt{\kappa})$ Hessian-vector calls. Additionally, this strategy avoids  $\fO\left(d_x^2+d_y^2\right)$ space to keep Hessian matrices.

\section{Complexity Comparison} \label{apsec:com}

Table \ref{table:compare} summarizes the theoretical results of proposed algorithms and existing methods. Note that the original analysis of GDA \cite{lin2019gradient} and PPA \cite{lin2020near} are based on the variable $\vy$ lies in a convex and compact constraint set. In fact, both of these analysis can be extended to unconstrained case easily. We give a brief sketch for the modification as follows.

\paragraph{GDA} 
Consider the proof in Section C.3 of \citet{lin2019gradient}. We only needs to keep the term $\delta_0$ and do not relax it into the diameter of the constraint set. Then it achieves $\fO\left((\kappa^2\ell+\kappa\ell^2)\eps^{-2}\right)$ gradient call upper bound for unconstrained case.

\paragraph{PPA}
Recall the proof in Section D.3 of \citet{lin2020near}. 
Let $\delta=\fO\left(\eps^2/(\kappa^2\ell)\right)$.
To achieve the corresponding result for unconstrained case in Table~\ref{table:compare}, we only needs to find $\vx_{t+1}$ such that
\begin{align*}
    \max_{\vy\in\BR^d} f(\vx_{t+1},\vy) + \ell\norm{\vx_{t+1}-\vx_t}^2 \leq \min_{\vx\in\BR^x}\left\{\max_{\vy\in\BR^y} f(\vx,\vy) + \ell\norm{\vx-\vx_t}^2\right\} + \delta
\end{align*}
in $\tilde\fO(\sqrt{\kappa}\log(1/\eps))$ gradient calls, which can be obtained by AL-SVRE with $n=1$~\cite{luo2021near}. 
More specifically, we require using AL-SVRE with $n=1$ to solve the following minimax problem
\begin{align}\label{prob:unbalanced-scsc}
    \min_{\vx\in\BR^x}\max_{\vy\in\BR^y} f_t(\vx,\vy) \triangleq f(\vx,\vy) + \ell\norm{\vx-\vx_t}^2.
\end{align}
Since we have supposed $f(\vx,\vy)$ is $\ell$-smooth and $\mu$-strongly-concave in $\vy$, the function $f_t(\vx,\vy)$ is $3\ell$-smooth, $\ell$-strongly-convex in $\vx$ and $\mu$-strongly-concave in $\vy$. Then the condition numbers of $f_t(\vx,\vy)$ are $\kappa_x=3\ell/\ell=3$ and $\kappa_y=3\ell/\mu=3\kappa$ respectively.
Note that the proof of Corollary 1 of \cite{luo2021near} does not depend on that $\vx$ or $\vy$ is constrained in a bounded set. Hence, applying
it with $n = 1$ means we need at most $\tilde\fO\left(\sqrt{\kappa}\log(1/\eps)\right)$ gradient calls to solve problem (\ref{prob:unbalanced-scsc}) with desired accuracy.

\begin{table}
\centering
\caption{We present the comparison of the algorithms for nonconvex-strongly-concave minimax optimization. Note that only the proposed MCN and IMCN guarantee to find the approximate second-order stationary point.}\label{table:compare}
\renewcommand{\arraystretch}{1.5}
{\small\begin{tabular}{c|c|c|c|c}
\hline
Algorithm & GDA~\cite{lin2019gradient} & PPA~\cite{lin2020near} & \!\!MCN~(Corollary~\ref{cor:complexity})\!\! & IMCN~(Corollary~\ref{cor:complexity2}) \\
\hline
$\norm{\nabla P(\vx)}\leq\eps$ & \Checkmark & \Checkmark & \Checkmark & \Checkmark \\
\!\!$\nabla^2 P(\vx) \succeq -\kappa^{1.5}\sqrt{\rho\eps}~\mI$\!\! & \XSolidBrush & \XSolidBrush & \Checkmark & \Checkmark  \\\hline
first-order oracle & \!\!$\fO\left((\kappa^2\ell+\kappa\ell^2)\eps^{-2}\right)$\!\! & \!\!$\tilde\fO\big(\sqrt{\kappa}\ell\eps^{-2}\big)$\!\! & $\tilde\fO\big(\kappa^2\sqrt{\rho}\eps^{-1.5}\big)$ & $\tilde\fO\big(\kappa^2\sqrt{\rho}\eps^{-1.5}\big)$ \\
second-order oracle & -- & -- & \!\!\!$\fO\big(\kappa^{1.5}\sqrt{\rho}\eps^{-1.5}\big)$\!\!\! &  --  \\
Hessian-vector oracle & -- & -- & -- &  $\tilde\fO\left(\kappa^{1.5}\ell\eps^{-2}\right)$ \\\hline
\end{tabular}}\vskip0.1cm
\end{table}

%%%%%%%%%%%%%%%%%%%%%%%%%%%%%%%%%%%%%%%%%%%%%%%%%%%%%%%%%%%%%%%%%%%%%%%%%%%%%%%
%%%%%%%%%%%%%%%%%%%%%%%%%%%%%%%%%%%%%%%%%%%%%%%%%%%%%%%%%%%%%%%%%%%%%%%%%%%%%%%

\section{Experimental Details}\label{apsec:exp}
In this section, we present the details of both synthetic minimax problem and the setting of domain adaptation experiments.

\subsection{Details of Synthetic Minimax Problem} \label{apsec:syn}

Our synthetic experiment are based on the following nonconvex-strongly-concave minimax problem 
\begin{align*}
\min_{\vx\in\BR^3} \max_{\vy\in\BR^2} f(\vx,\vy) =
w(x_3) - \frac{y_1^2}{40} + x_1 y_1 - \frac{5y_2^2}{2} + x_2y_2,
\end{align*}
where $\vx=[x_1, x_2, x_3]^\top$, $\vy=[y_1, y_2]^\top$ and
\begin{align*}
w(x) =
\begin{dcases}
\sqrt{\epsilon} \left(x + (L+1) \sqrt{\epsilon}\right)^2 - \frac{1}{3} \left(x + (L+1) \sqrt{\epsilon}\right)^3 - \frac{1}{3} (3 L + 1) \epsilon^{3/2} ,
& x \le -L \sqrt{\epsilon} ; \\
\epsilon x + \frac{\epsilon^{3/2}}{3} ,
& -L \sqrt{\epsilon} < x \le -\sqrt{\epsilon} ; \\
-\sqrt{\epsilon} x^2 - \frac{x^3}{3} ,
& -\sqrt{\epsilon} < x \le 0 ; \\
-\sqrt{\epsilon} x^2 + \frac{x^3}{3} ,
& 0 < x \le \sqrt{\epsilon} ; \\
-\epsilon x + \frac{\epsilon^{3/2}}{3} ,
& \sqrt{\epsilon} < x \le L \sqrt{\epsilon} ; \\
\sqrt{\epsilon} \left(x - (L+1) \sqrt{\epsilon}\right)^2 + \frac{1}{3} \left(x - (L+1) \sqrt{\epsilon}\right)^3 - \frac{1}{3} (3 L + 1) \epsilon^{3/2} ,
& L \sqrt{\epsilon} \le x . \\
\end{dcases}
\end{align*}

The nonconvex function $w(x)$ is designed by~\citet{tripuraneni2017stochastic} and we set $\epsilon = 0.01$ and $L = 5$ in our experiment. 
We can verify that 
\begin{align*}
\vy^*(\vx) = \begin{bmatrix}
20x_1 \\[0.1cm] \dfrac{x_2}{5}
\end{bmatrix}
\quad \text{and}  \quad
P(x) = w(x_3) +  10x_1^2 + \frac{1}{10}x_2^2.
\end{align*}

\begin{figure}[t]
\centering
\includegraphics[width=0.45\textwidth]{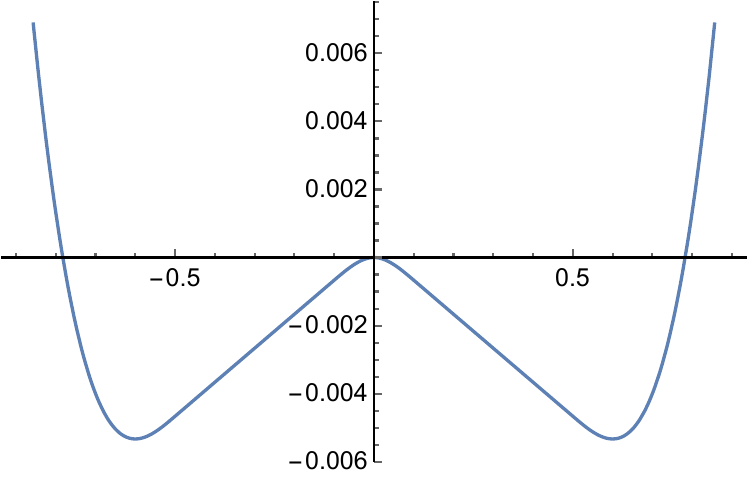}
\caption{\citet{tripuraneni2017stochastic} designed the W-shaped function and provided this figure for visualization.}
\label{fig:w}
\end{figure}

\subsection{Domain Adaptation} \label{apsec:real}
The experiments are conducted on a workstation with Intel Xeon 2.6GHz CPU, 256GB memory and one Nvidia Tesla V100 GPU. 
We implement the algorithms by using Pytorch 1.10.1 and Python 3.8.8. We choose $\alpha=1$ and $\lambda = 0.2$ for the model.

The learning rate of GDA and AGD is selected from $\left\{c\cdot 10^{-i}:c\in\{1,5\},i\in \{1,2,3,4,5\}\right\}$. In the implementation of the GDA, the learning rate of $\vx$ and $\vy$ are chosen separately. For IMCN algorithm, the parameter $M$ is selected from $\left\{c\cdot 10^{-i}:c\in\{2,5\},i\in \{0,1,2\}\right\}$. We set the threshold of using Cauchy point (i.e., Line 2 of Algorithm \ref{alg:cubic-GD}) as $0.01$. The number of AGD iterations and Cubic-Solver iterations are selected from $\{50,100,150,200\}$. The learning rate of cubic solver is selected from $\{0.2,0.02,0.002\}$.

\end{document}